\documentclass[a4paper]{amsart}
\usepackage{amsmath,amssymb,times, amscd, graphicx, xspace,mathrsfs,a4wide,hyperref,url}
\usepackage[title]{appendix}
\hypersetup{
  colorlinks   = true, 
  urlcolor     = blue, 
  linkcolor    = blue, 
  citecolor   = blue 
}

\newcommand{\Z}{\ensuremath{\mathbb{Z}}\xspace}
\newcommand{\Q}{\ensuremath{\mathbb{Q}}\xspace}
\newcommand{\R}{\ensuremath{\mathbb{R}}\xspace}
\newcommand{\C}{\ensuremath{\mathbb{C}}\xspace}
\newcommand{\A}{\ensuremath{\mathbb{A}}\xspace}
\newcommand{\F}{\ensuremath{\mathbb{F}}\xspace}
\newcommand{\Qp}{\ensuremath{\mathbb{Q}_{p}}\xspace}
\newcommand{\Zp}{\ensuremath{\mathbb{Z}_{p}}\xspace}
\newcommand{\Fp}{\ensuremath{\mathbb{F}_{p}}\xspace}

\newcommand{\q}{\ensuremath{\mathfrak{q}}\xspace}

\newcommand{\comment}[1]{}

\DeclareMathOperator{\Gal}{Gal}
\DeclareMathOperator{\End}{End}
\DeclareMathOperator{\Hom}{Hom}
\DeclareMathOperator{\Lie}{Lie}

\DeclareMathOperator{\Spec}{Spec}
\DeclareMathOperator{\Spf}{Spf}
\DeclareMathOperator{\Spd}{Spd}

\DeclareMathOperator{\Spa}{Spa}
\DeclareMathOperator{\Tor}{Tor}

\DeclareMathOperator{\Ker}{Ker}

\newcommand{\GL}{\ensuremath{\mathrm{GL}}\xspace}

\newcommand{\mbf}{\mathbf}
\newcommand{\mb}{\mathbb}
\newcommand{\mc}{\mathcal}

\newcommand{\mf}{\mathfrak}
\newcommand{\vp}{\varpi}
\newcommand{\ra}{\rightarrow}

\newcommand{\sub}{\subseteq}
\newcommand{\oo}{\mathcal{O}}
\newcommand{\ol}{\overline}
\newcommand{\ul}{\underline}

\newcommand{\ka}{\kappa}

\newcommand{\wh}{\widehat}
\newcommand{\wt}{\widetilde}
\newcommand{\ok}{\mathcal{O}_{K}}
\newcommand{\Xf}{\mathfrak{X}}
\newcommand{\Af}{\mathfrak{A}}
\newcommand{\Xb}{\overline{X}}
\newcommand{\Ab}{\overline{A}}
\newcommand{\Xc}{\mathcal{X}}
\newcommand{\Ac}{\mathcal{A}}

\newcommand{\G}{\mathcal{G}}

\newcommand{\Pro}{\mathbb{P}^{n-1}}
\newcommand{\V}{\mathcal{V}}

\def\et{\mathrm{\acute{e}t}}

\def\GH{\mathrm{GH}}

\newtheorem{thmx}{Theorem}
\newtheorem{theorem}{Theorem}[subsection]
\newtheorem{proposition}[theorem]{Proposition}
\newtheorem{corollary}[theorem]{Corollary}
\newtheorem{lemma}[theorem]{Lemma}

\theoremstyle{definition}
\newtheorem{definition}[theorem]{Definition}

\newtheorem{remark}[theorem]{Remark}

\mathchardef\mhyphen="2D
\title{}
\author{}

\input xy
\xyoption{all}
\usepackage{amscd}

\usepackage{color}

\begin{document}

\title{A quotient of the Lubin--Tate tower II}
\author{Christian Johansson and Judith Ludwig \\ With an appendix by David Hansen}

\address{Department of Mathematical Sciences, Chalmers University of Technology and the University of Gothenburg, 412 96 Gothenburg, Sweden}
\email{chrjohv@chalmers.se}

\address{IWR, University of Heidelberg, Im Neuenheimer Feld 205, 69120 Heidelberg, Germany}
\email{judith.ludwig@iwr.uni-heidelberg.de}

\address{Max Planck Institute for Mathematics, Vivatsgasse 7, 53111 Bonn, Germany}
\email{dhansen@mpim-bonn.mpg.de}

\maketitle

\begin{abstract}
In this article we construct the quotient $\mc{M}_\mbf{1}/P(K)$ of the infinite-level Lubin--Tate space $\mc{M}_\mbf{1}$ by the parabolic subgroup $P(K) \subset \GL_n(K)$ of block form $(n-1,1)$ as a perfectoid space, generalizing the results of \cite{lud} to arbitrary $n$ and $K/\Q_p$ finite. For this we prove some perfectoidness results for certain Harris--Taylor Shimura varieties at infinite level.  
As an application of the quotient construction we show a vanishing theorem for Scholze's candidate for the mod $p$ Jacquet--Langlands and mod $p$ local Langlands correspondence. An appendix by David Hansen gives a local proof of perfectoidness of $\mc{M}_\mbf{1}/P(K)$ when $n=2$, and shows that $\mc{M}_\mbf{1}/Q(K)$ is not perfectoid for maximal parabolics $Q$ not conjugate to $P$. 
\end{abstract}

\counterwithin{equation}{subsection}

\section{Introduction}

This article generalises the main results of \cite{lud}. Let $K/\mathbb{Q}_p$ be a finite extension with ring of integers $\mc{O}_K$, uniformizer $\varpi$ and residue field $k$. Fix an algebraically closed and complete non-archimedean field $C$ containing $K$.
Let $\mc{M}_\mbf{1}$ denote the infinite-level Lubin--Tate space over $C$. By work of Weinstein, $\mc{M}_\mbf{1}$ is a perfectoid space equipped with an action of $\GL_n(K)$. Let $P\subset \GL_n$ be the parabolic subgroup consisting of upper triangular block matrices of block size $(n-1,1)$. In this article we prove the following theorem.

\begin{thmx}\label{A}
The quotient $\mc{M}_{P(K)}:= \mc{M}_\mbf{1}/P(K)$ is a perfectoid space over $C$ of Krull dimension $n-1$. 
\end{thmx}

Here we take the quotient in Huber's category $\mc{V}$ of locally v-ringed spaces, as in \cite{lud}. The construction of the perfectoid structure on $\mc{M}_{P(K)}$ follows the strategy via globalisation from \cite{lud}, where the quotient was constructed in the case when $n=2$ and $K=\mathbb{Q}_p$. In that case, modular curves were used to globalise and one could rely on the perfectoidness results of \cite{tor}. For our generalisation we make use of the Shimura varieties studied by Harris--Taylor in their proof of the local Langlands correspondence for $\GL_n$ \cite{ht}, and this necessitates some new perfectoidness results.

\medskip

Let us now describe the strategy of \cite{lud} and this paper in slightly more detail; the reader may also consult the introduction to \cite{lud}. The space $\mc{M}_\mbf{1}$ has a $\GL_n(\mc{O}_K)$-equivariant decomposition $\mc{M}_\mbf{1} \cong \bigsqcup_{i\in \mathbb{Z}} \mc{M}^{(i)}_\mbf{1}$ into pairwise isomorphic spaces (coming from the decomposition of the Lubin--Tate space at level $0$ into connected components). As in \cite{lud} we reduce the construction of $\mc{M}_\mbf{1}/P(K)$ to the construction of $\mc{M}^{(0)}_\mbf{1}/P(\mc{O}_K)$ using the geometry of the Gross--Hopkins period map. We can realize $\mc{M}^{(0)}_\mbf{1}$ as an open subspace of a certain infinite level perfectoid Harris--Taylor Shimura variety $\mc{X}_\mbf{1}$. The image lands inside what we call the ``complementary locus'' $\mc{X}^{comp}_\mbf{1}$, which is a subspace of $\mc{X}_\mbf{1}$ defined in terms of the Hodge--Tate period map. We show that the quotient $\mc{X}^{comp}_\mbf{1}/P(\mc{O}_K)$ exists and is perfectoid, and existence and perfectoidness of $\mc{M}^{(0)}_\mbf{1}/P(\mc{O}_K)$ is then a direct consequence. The main ingredient of the proof is the construction of a perfectoid overconvergent anticanonical tower for our Harris--Taylor Shimura varieties (analogous to \cite[Corollary 3.2.20]{tor}), and this forms the technical heart of this paper.

\medskip

Theorem \ref{A} has the following application. Let $D^\times$ be the group of units in the central division algebra $D$ over $K$ with invariant $1/n$. In \cite{plt}, Scholze constructs a functor that is expected to be simultaneously related to a conjectural mod $p$ local Langlands correspondence for the group $\GL_n(K)$ and an equally conjectural mod $p$ Jacquet--Langlands transfer between $\GL_n(K)$ and $D^\times$. 
For any admissible smooth representation $\pi$ of $\GL_n(K)$ on a $\mathbb{F}_p$-vector space, Scholze constructs an \'etale sheaf $\mc{F}_\pi$ on $\mathbb{P}^{n-1}$ using the Gross--Hopkins period morphism $\mc{M}_\mbf{1} \rightarrow \mathbb{P}^{n-1}$. The cohomology groups 
\[\mathcal{S}^i(\pi):= H_\et^i(\mathbb{P}^{n-1},\mc{F}_\pi), \ i \geq 0,\]
are admissible $D^\times$-representations which carry an action of $\Gal(\overline{K}/K)$ and vanish in degree $i > 2(n-1)$ (\cite[Theorem 1.1]{plt}). As an application of the construction of $\mc{M}_{P(K)}$ we prove the following vanishing result.

\begin{thmx}[Theorem \ref{vanish}]\label{thm B} Let $P^*\subset \GL_n$ be a parabolic subgroup contained in $P$ and let $\sigma$ be a smooth admissible representation of $P^*(K)$ with parabolic induction $\pi:= \mathrm{Ind}^{\GL_n(K)}_{P^*(K)}{\sigma}$ to $\GL_n(K)$. Then
\[\mathcal{S}^i(\pi)=0 \ \text{ for all } i> n-1.\]
\end{thmx}

This theorem generalises \cite[Theorem 4.6]{lud}, which is the special case when $n=2$, $K=\Qp$ and $\sigma$ is a character.

\medskip

The paper also features an appendix, written by David Hansen, in which the space $\mc{M}_{P(K)}$ is studied in Scholze's category of diamonds from a purely local point of view, using the moduli-theoretic description of $\mc{M}_\mbf{1}$ due to Scholze--Weinstein in terms of vector bundles on the Fargues--Fontaine curve. The main results of the appendix can be summarized as follows; we refer to the introduction to the appendix for further details.

\begin{thmx}[Hansen; Corollary \ref{0.4} and Theorem \ref{0.5}]\label{C} Let $Q \subset \GL_n$ be a standard (block upper triangular) maximal parabolic subgroup of $\GL_n$. Then the diamond quotient $\mc{M}_\mbf{1}/\ul{Q(K)}$ is proper and $\ell$-cohomologically smooth (in the sense of \cite{dia}) for all primes $\ell \neq p$, but not a perfectoid space if $Q \neq P$. Moreover, in the special case $n=2$, Theorem \ref{A} may be proved by purely local methods\footnote{We remark that if either $\mc{M}_{P(K)}$ or $\mc{M}_\mbf{1}/\ul{P(K)}$ is perfectoid, then they are equal. In particular, Theorems \ref{A} and \ref{C} a posteriori concern the same space (when $P=Q$).}.
\end{thmx}

\medskip

Let us now describe the contents of this paper. Sections \ref{sec2} and \ref{sec3} are devoted to proving the perfectoidness results for the Harris--Taylor Shimura vartieties that we need. While it might be possible to deduce what we need from \cite{tor}, certain technicalities made such an approach seem very cumbersome and unsatisfactory to us. We have therefore elected to construct the anticanonical tower in the Harris--Taylor setting directly, following the approach in \cite{tor} (simplified by the absence of a boundary). Scholze's approach relies on an integral theory of canonical subgroups and on the Hasse invariant, so we need a version of these notions for our Harris--Taylor Shimura varieties (which have empty ordinary locus in general). Section \ref{sec2} develops a theory of $\mu$-ordinary Hasse invariants and canonical subgroups for one-dimensional compatible Barsotti--Tate $\mc{O}_K$-modules $\mathcal{G}/S$ of height $n$, where $S$ is a $k$-scheme. We use a Hasse invariant due to Ito \cite{ito1,ito2} which turns out to be perfect for adapting Scholze's approach to canonical subgroups based on Illusie's deformation theory for group schemes \cite{ill}. We refer to Remark \ref{strata} for further discussion of the Hasse invariants used in this paper.

\medskip

Equipped with a theory of canonical subgroups, Section \ref{sec3} proceeds to construct the $\epsilon$-neighbourhoods of the anticanonical tower in our setting. It is a tower of formal schemes $(\widehat{\mathfrak{X}}(\epsilon)_{m,a})_{m\geq 0}$ whose generic fibres $\mc{X}(\epsilon)_{m,a}$ embed into the adic Shimura varieties $\mc{X}_{U_0(\varpi^m)}$, where the level at the important prime is  $ U_{0}(\vp^{m}) := \{ g\in \GL_{n}(\ok) \mid g\,\, {\rm mod}\,\, \vp^{m} \in P(\ok/\vp^{m}) \}$; we refer to the main body of the paper for precise definitions. In the limit we get a perfectoid space (Theorem \ref{epsilonnhood}). This then allows us to prove the analogues of the main geometric results of \cite{tor}, importantly including the construction of a Hodge--Tate period map $\pi_{HT}:\mc{X}_\mbf{1} \ra \Pro$ (see Theorem \ref{full}). For this we have found it convenient to use the language of diamonds \cite{dia}. We end Section \ref{sec3} by using the geometry of the Hodge--Tate period map to show that the quotient $\mc{X}_\mbf{1}^{comp}/P(\mc{O}_K)$ of the complementary locus is perfectoid (Theorem \ref{comp locus at gamma 0 level}).

\medskip

Section \ref{sec:LT} then uses the results of Section \ref{sec3} to prove Theorem \ref{A} and deduce some properties of the space $\mc{M}_{P(K)}$. The Gross--Hopkins period map plays a prominent role in the proofs, and it induces a quasicompact map $\ol{\pi}_{GH} : \mc{M}_{P(K)} \ra \Pro$.

\medskip

The main part of the paper then finishes with section \ref{sec:app}, which proves Theorem \ref{thm B}. The calculations follow the same path as Section 4 of \cite{lud}, the idea being that pushforward along the map $\ol{\pi}_{GH} : \mc{M}_{P(K)} \ra \Pro$ is a geometric realisation of the parabolic induction functor, so \'etale cohomology of $\mc{F}_\pi$ on $\Pro$ is equal to \'etale cohomology of an analogously defined sheaf $\mc{F}_\sigma$ on $\mc{M}_{P(K)}$. For the reader familiar with \cite{lud}, we mention that our argument deviates somewhat from that of \cite{lud}. The most important point is that, by invoking a general result of Scheiderer \cite{dim} on the cohomological dimension of spectral spaces, it suffices for us to relate the \'etale cohomology of $\mc{F}_\sigma$ on $\mc{M}_{P(K)}$ to an analytic cohomology group on $\mc{M}_{P(K)}$. In \cite{lud} it was instead related to an analytic cohomology group on $\mathbb{P}^{1}$, which necessitated the study of the fibres of $\ol{\pi}_{GH}$. Moreover, to deal with the fact that $\sigma$ will typically be infinite-dimensional, we use some additional limit arguments. 

\medskip

The paper then finishes with Hansen's appendix; we refer to its introduction for a detailed overview of its results and methods.

\medskip

\subsection*{Acknowledgments} The authors wish to thank David Hansen for his interest in our work and for a number of discussions relating to it, and for generously agreeing to include his local study of $\mc{M}_{P(K)}$ in an appendix to our paper. The authors would furthermore like to thank Florian Herzig and Vytautas Pa\v{s}k\={u}nas for helpful conversations, and an anonymous referee for useful comments and corrections. C.J. also wishes to thank Daniel Gulotta, Chi-Yun Hsu, Lucia Mocz, Emanuel Reinecke, Sheng-Chi Shih and especially Ana Caraiani for all discussions relating to \cite{arizona}, which have had a large influence on this paper. C.J. was supported by the Herchel Smith Foundation for part of the work on this paper. J.L.\ was supported by the Max Planck Institute for Mathematics and the IWR Heidelberg. 

\section{Hasse invariants and canonical subgroups}\label{sec2}

\subsection{Global setup}

We start by introducing some notation which will be in place throughout the paper. Fix, once and for all, a prime $p$ and an integer $n\geq 2$. We also fix a finite extension $K/\Qp$ with ring of integers $\ok$, uniformizer $\varpi$, residue field $k$, ramification index $e$, and inertia degree $f$. Set $q=p^{f}$. As in \cite{ht}, we choose a totally real field $F^{+}$ of degree $d$, with primes $v=v_{1},v_{2},\dots,v_{r}$ above $p$, such that $F^{+}_{v}\cong K$ (we fix such an isomorphism and think of it as an equality). We then choose an imaginary quadratic field $E$ in which $p$ splits as $p=uu^{c}$, where $c$ denotes complex conjugation, and let $F=EF^{+}$; this is a CM field. We let $w_{i}$, $i=1,\dots,r$, denote the unique prime in $F$ above $u$ and $v_{i}$, and put $w=w_{1}$. 

\medskip
Let us now recall the setup of \cite{ht}, to which we refer for more details. Following \cite[\S I.7]{ht}, we let $B/F$ denote a central division algebra of dimension $n^2$ such that

\begin{itemize}
\item The opposite algebra $B^{op}$ is isomorphic to $B\otimes_{F,c}F$;

\item $B$ is split at $w$;

\item if $x$ is a place of $F$ whose restriction to $F^+$ does not split in $F/F^{+}$, $B_{x}$ is split;

\item if $x$ is a place of $F$ whose restriction to $F^+$ splits in $F/F^{+}$, $B_{x}$ is either split or a division algebra;

\item if $n$ is even, then the number of finite places of $F^{+}$ above which $B$ is ramified is congruent to $1+dn/2$ modulo $2$.
\end{itemize}

Choose an involution $\ast$ of the second kind on $B$. Let $V=B$ and consider it as a $B\otimes_{F}B^{op}$-module. For any $\beta\in B$ with $\beta^{\ast}=-\beta$, we can define an alternating $\ast$-Hermitian pairing $V\times V \ra \Q$ by
$$ (x,y)={\rm tr}_{F}{\rm tr}_{B/F}(x\beta y^{\ast}) $$
where ${\rm tr}_{B/F}$ denotes the reduced trace. We fix a $\beta \in B$ with $\beta^{\ast}=-\beta$. We define another involution $\#$ of the second kind on $V$ by $x^{\#}= \beta x^{\ast} \beta^{-1}$. We let $G/\Q$ be the reductive group with the functor of points ($R$ any $\Q$-algebra)
$$ G(R)=\left\{(g,\lambda) \in (B^{op}\otimes_{\Q}R)^{\times} \times R^{\times} \mid gg^{\#} = \lambda \right\}. $$
The map $(g,\lambda) \mapsto \lambda $ defines a homomorphism $\nu :G \ra \mb{G}_{m}$ (the similitude factor) and we denote its kernel by $G_{1}$. If $x$ is a prime in $\Q$ which splits as $x=yy^{c}$ in $E$, then $y$ induces an isomorphism $G(\Q_{x})\cong (B_{y}^{op})^{\times} \times \Q_{x}^{\times}$. In particular, $u$ induces an isomorphism
$$ G(\Q_{p}) \cong (B_{u}^{op})^{\times} \times \Qp^{\times} \cong \Qp^{\times} \times \prod_{i=1}^{r}(B_{w_{i}}^{op})^{\times}. $$
We will assume (see \cite[Lemma I.7.1]{ht} and the discussion following it; we assume that $\beta$ is chosen so that this applies) that
\begin{itemize}
\item if $x$ is a prime in $\Q$ which does not split in $E$, then $G\times \Q_{x}$ is quasi-split;
\item the pairing $(-,-)$ on $V\otimes_{\Q}\R$ has invariants $(1,n-1)$ at one embedding $F^{+} \hookrightarrow \R$ and $(0,n)$ at all other embeddings $F^{+} \hookrightarrow \R$. 
\end{itemize}

Next, fix a maximal order $\Lambda_{i}=\oo_{B,w_{i}}\sub B_{w_{i}}$ for each $i=1,\dots,r$. The pairing $(-,-)$ gives a perfect duality between $V_{w_{i}}:=V_{\otimes_{F}}F_{w_{i}}$ and $V_{w_{i}^{c}}$, and we let $\Lambda_{i}^{\vee}\sub V_{w_{i}^{c}}$ denote the dual of $\Lambda_{i}$. We get a $\Zp$-lattice
$$ \Lambda = \bigoplus_{i=1}^{r}\Lambda_{i} \oplus \bigoplus_{i=1}^{r}\Lambda_{i}^{\vee} \sub V\otimes_{\Q}\Qp. $$
and $(-,-)$ restricts to a perfect pairing $\Lambda \times \Lambda \ra \Zp$. We fix an isomorphism $\oo_{B,w}\cong M_{n}(\oo_{K})$, and we compose it with the transpose map to get an isomorphism $\oo_{B,w}^{op}\cong M_{n}(\ok)$. If $\epsilon\in M_{n}(\oo_{F,w})$ is the idempotent which has $1$ in the $(1,1)$-entry and $0$ everywhere else; $\epsilon$ induces an isomorphism 
$$ \Lambda_{11}:=\epsilon \oo_{B,w} \cong (\oo_{K}^{n})^{\vee}. $$
Finally, we let $\oo_{B}$ denote the unique maximal $\Z_{(p)}$-order in $B$ which localizes to $\oo_{B,w_{i}}$ for all $i$ and satisfies $\oo_{B}^{\ast}=\oo_{B}$ (see \cite[p. 56-57]{ht} for further discussion).

\medskip
Let us now recall the integral models of the Shimura varieties for $G$; we refer to \cite[\S III.4]{ht} for more details. We remark that we will only need integral models in the case $m_1=0$ below, when the models are smooth, but we recall the definitions in the general case. If $S$ is an $\ok$-scheme and $A/S$ is an abelian scheme with an injective homomorphism $i : \oo_{B} \hookrightarrow \End(A)\otimes_{\Z}{\Z_{(p)}}$, we write $\G_{A}$ for the $p$-divisible group
$$ \G_{A} := \epsilon A[\varpi^{\infty}]. $$
Fix a sufficiently small compact open subgroup $U^p \sub G(\A^{\infty,p})$ and a tuple $\ul{m}=(m_{1},\dots,m_{r})\in \Z_{\geq 0}^{r}$. The moduli functor $\Xf_{\ul{m}}$ (we suppress $U^{p}$ from the notation) is defined as follows:
If $S$ is a connected locally noetherian $\ok$-scheme and $s$ is a geometric point of $S$, $\Xf_{\ul{m}}$ is the set of equivalence classes of $(r+4)$-tuples $(A,\lambda,i,\ol{\eta}^p,\alpha_{i})$ where
\begin{itemize}
\item $A/S$ is an abelian scheme of dimension $dn^2$;

\item $\lambda : A \ra A^{\vee}$ is a prime-to-$p$ polarization;

\item $i : \oo_{B} \hookrightarrow \End(A)\otimes_{\Z}\Z_{(p)}$ is a homomorphism such that $(A,i)$ is compatible and $\lambda \circ i(b) = i(b^{\ast})^{\vee} \circ \lambda$ for all $b\in \oo_B$;

\item $\ol{\eta}^{p}$ is a $\pi_{1}(S,s)$-invariant $U^p$-orbit of isomorphisms of $B\otimes_{\Q}\A^{\infty,p}$-modules $\eta^p : V \otimes_{\Q}\A^{\infty,p} \ra V^{p}A_{s}$ which take the pairing $(-,-)$  on $V\otimes_{\Q}\A^{\infty,p}$ to a $(\A^{\infty,p})^{\times}$-multiple of the $\lambda$-Weil pairing on $V^{p}A_{s}$;

\item $\alpha_{1} : \varpi^{-m_{1}}\Lambda_{11}/\Lambda_{11} \ra \G_{A}[\varpi^{m_{1}}]$ is a Drinfeld $\varpi^{m_{1}}$-level structure;

\item for $i=2,\dots,r$, $\alpha_{i} : (\varpi_{i}^{-m_{i}}\Lambda_{i}/\Lambda_{i})_{S} \ra A[\varpi_{i}^{m_{i}}]$ is an isomorphism of $S$-schemes with $\oo_B$-actions.
\end{itemize}
Here $\varpi=\varpi_{1},\dots,\varpi_{r}$ are uniformizers of $\oo_{F,w_{i}}$. Two $(r+4)$-tuples $(A,\lambda,i,\ol{\eta}^{p},\alpha_{i})$ and $(A^{\prime},\lambda^{\prime},i^{\prime},(\ol{\eta}^{p})^{\prime},\alpha_{i}^{\prime})$ are equivalent if there is a prime-to-$p$ isogeny $\delta : A \ra A^{\prime}$ and a $\gamma \in \Z_{(p)}^{\times}$ such that $\delta$ carries $\lambda$ to $\gamma \lambda^{\prime}$, $i$ to $i^{\prime}$, $\ol{\eta}^{p}$ to $(\ol{\eta}^{p})^{\prime}$, and $\alpha_{i}$ to $\alpha^{\prime}_{i}$. $\Xf_{\ul{m}}(S,s)$ is canonically independent of the choice of $s$, and we get a functor on all locally noetherian $\ok$-schemes by requiring that
$$ \Xf_{\ul{m}}\left( \coprod_{i} S_{i} \right) = \prod_{i} \Xf_{\ul{m}}(S_{i}). $$
This functor is representable by a projective scheme over $\ok$, which is smooth when $m_{1}=0$. By abuse of notation, we will denote it by $\Xf_{\ul{m}}$. If $\ul{m}^{\prime}\geq \ul{m}$ (by which we mean $m_{i}^{\prime} \geq m_{i}$ for all $i$), then the natural map $\Xf_{\ul{m}^{\prime}} \ra \Xf_{\ul{m}}$ is finite and flat; moreover it is \'etale if $m_{1}^{\prime}=m_{1}$. See \cite[pp. 109--112]{ht}. We will denote the special fibre of $\Xf_{\ul{m}}$ by $\Xb_{\ul{m}}$, and the generic fibre by $X_{\ul{m}}$. Over $\Xb_{\ul{m}}$, we have a universal abelian scheme $\ol{A}_{\ul{m}}$ and the associated Barsotti--Tate $\ok$-module $\G_{\ol{A}_{\ul{m}}}$, which we will denote by $\G_{\ul{m}}$ or just $\G$ if the context is clear. One defines a locally closed subscheme $\Xb_{\ul{m}}^{(h)}$ by requiring that the \'etale part $\G^{et}$ of $\G$ has constant $\ok$-height $h$, where $0\leq h \leq n-1$. Then $\Xb_{\ul{m}}^{(h)}$ is smooth of pure dimension $h$ \cite[Corollary III.4.4]{ht}.

\subsection{Hasse invariants}

In this section, we let $S$ be a scheme over $k$ and we let $\G/S$ be a compatible Barsotti--Tate $\ok$-module of dimension $1$ and height $n$ (throughout this article, heights are $\ok$-heights unless otherwise specified). Let us briefly recall the notion of compatibility, referring to \cite[p. 59]{ht} for more details. The Lie algebra $\Lie \G$ of $\G$ is locally free $\oo_S$-module and a priori carries two actions of $\ok$. One comes from the $\ok$-action map $\ok \to \End_S(\G)$, and the second action comes from the natural map $\ok \to k \to \oo_S$ together with the $\oo_S$-module structure on $\Lie \G$. Compatibility then means that these two $\ok$-actions agree. The goal of this section is to define a so-called $\mu$-ordinary Hasse invariant for $\G/S$. The topic of generalized Hasse invariants has received a lot of attention recently. In the case of $\mu$-ordinary Hasse invariants we mention the works \cite{gn,kw,he,bh}; moreover the works \cite{gk,box} construct generalized Hasse invariants on all Ekedahl--Oort strata (in the cases when they apply). In particular, $\mu$-ordinary Hasse invariants have been defined in large generality (including the cases needed here) by Bijakowski and Hernandez \cite{bh}. We have nevertheless opted for a direct approach. It should be noted that `Hasse invariants', as the term exists in the literature, are not unique. The definition given here is chosen because it is very well suited for adapting Scholze's approach to the canonical subgroup to the situation of our Harris--Taylor Shimura varieties, which is the topic of the next subsection. After writing a first draft of this section, we learnt that the definition of a $\mu$-ordinary Hasse invariant we give here was first given by Ito \cite{ito1, ito2}. Since we are not aware of any detailed account of Ito's construction, we give our construction (it seems very likely that they are the same, judging from the sketch in \cite{ito1}). Ito did not only construct a $\mu$-ordinary Hasse invariant but also `strata' Hasse invariants on Harris--Taylor Shimura varieties, and the construction below can easily be adapted to produce such Hasse invariants (see Remark \ref{strata}).

\medskip
We start with a description of some Dieudonn\'e modules. Let $\ka$ be an algebraically closed field containing $k$ and assume that $S=\Spec \ka$. Then we have $\G\cong \G^{et}\times \G^{0}$ (\'etale and connected parts) and both $\G^0$ and $\G^{et}$ are Barsotti--Tate $\ok$-modules. Let $h$ be the height of $\G^{et}$, then $0\leq h \leq n-1$. By the Dieudonn\'e--Manin theorem, $\G^{et}$ and $\G^{0}$ are determined up to isomorphism by their Dieudonn\'e modules. The Dieudonn\'e module of $\G^{0}$ is isomorphic to a Dieudonn\'e module $M_{n-h}$, which we now describe. We write $W(\ka)$ for the Witt vectors of $\ka$, and $\sigma$ for the lift of the $p$-th power Frobenius. $M_{n-h}$ has a Frobenius $F$ and a Verschiebung $V$, and has a basis $\omega,F\omega,F^{2}\omega,\dots,F^{n-h-1}\omega$ over $\ok\otimes_{\Zp}W(\ka)$, i.e. 
$$ M_{n-h} = \bigoplus_{i=0}^{n-h-1} (\ok \otimes_{\Zp}W(\ka)).F^{i}\omega. $$
To finish the description, we need to describe $F$, and we know that it is $\sigma$-linear and it sends $F^{i}\omega$ to $F^{i+1}\omega$ for $i=0,\dots,n-h-1$, so it remains to determine $F^{n-h}\omega$. For this, we write
$$ \ok \otimes_{\Zp} W(\ka) = \bigoplus_{\tau \in T} \ok \otimes_{\oo_{K_{0}},\iota \circ \tau}W(\ka), $$
where $T=\Gal(k/\Fp)$, $K_{0}$ is the maximal unramified subextension of $K/\Qp$, and $\iota : \oo_{K_{0}} \hookrightarrow W(\ka)$ is the lift of the inclusion $k\sub \ka$. Then 
$$ M_{n-h} = \bigoplus_{i=0}^{n-h-1} \bigoplus_{\tau \in T} (\ok \otimes_{\oo_{K_{0}},\iota \circ \tau}W(\ka)).F^{i}\omega. $$
We then define
$$ F^{n-h}\omega = (a_{\tau})_{\tau}\omega, $$
where $a_{id}=\vp \otimes 1$ and $a_{\tau}=1\otimes 1$ if $\tau\neq id$. $V$ is then defined uniquely by the condition $FV=VF=p$. The Dieudonn\'e module of $\G^{et}$ is
$$ (\ok\otimes_{\Zp}W(\ka))^{h} $$
with $F$ acting as $x \otimes y \mapsto x \otimes \sigma(y)$ on every factor. Taking the direct sum gives us the Dieudonn\'e module of $\G$.

\begin{definition}
Let $S=\Spec \ka$, where $\ka \supseteq k$ algebraically closed. We say that $\G$ is \emph{$\mu$-ordinary} if $\G^{et}$ has height $n-1$. For a general $S/k$ and $\G/S$, we say that $\G$ is $\mu$-ordinary if $\G_{x}$ is $\mu$-ordinary for every geometric point $x$ of $S$.  
\end{definition}

We now give an axiomatic definition of the $\mu$-ordinary Hasse invariant. Here and elsewhere we use the following piece of notation: For any integer $m\geq 1$, the twist $\G^{(q^m)}$ is defined as the pullback of $\G$ along the absolute $q^m$-th power Frobenius $F_{q^m}: S \ra S$. The relative $q^m$-power Frobenius will be denoted by $Fr_{q^m}$.

\begin{definition} Let $S/k$ be a scheme and let $\G/S$ be a one-dimensional compatible Barsotti--Tate $\ok$-module of height $n$.
\begin{enumerate}
\item If $\varpi : \G \ra \G$ factors through $Fr_{q} : \G \ra \G^{(q)}$, then we denote by $\ol{V}$ the unique isogeny $\G^{(q)} \ra \G$ such that $\ol{V} \circ Fr_{q} =\varpi$.

\item In the situation in (1), $\ol{V}$ induces a pullback map $\ol{V}^{\ast} : \omega_{\G} \ra \omega_{\G^{(q)}}\cong \omega_{\G}^{\otimes q}$ on top differentials, which corresponds to an element $H\in H^{0}(S,\omega_{\G}^{q-1})$. We define $H$ to be the \emph{$\mu$-ordinary Hasse invariant}.
\end{enumerate}
\end{definition}

The following proposition shows that we have $\mu$-ordinary Hasse invariants whenever $S$ is reduced.

\begin{proposition}\label{factoring}
Let $S$ be a reduced scheme over $k$ and $\G/S$ a one-dimensional compatible Barsotti--Tate $\ok$-module of height $n$. Then the isogeny $\varpi : \G \ra \G$ factors through the $q$-th power Frobenius isogeny $Fr_{q}: \G \ra \G^{(q)}$.
\end{proposition}

\begin{proof}
The proposition is equivalent to showing that $\Ker Fr_{q} \sub\Ker \varpi =\G[\varpi]$. Both are finite locally free subschemes of the finite locally free scheme $\G[q]$, so we are in the situation where we have a finite locally free scheme $G$ over a reduced $k$-scheme $S$, and two finite locally free subschemes $H,K\sub G$ and we want to show that $H\sub K$. We claim that it is enough to check this on geometric points.

\medskip
To see this we argue as follows. First, it is enough to check it Zariski-locally on $S$. So without loss of generality $S=\Spec(A)$ is affine, and $G=\Spec(B)$ where $A \ra B$ is projective; moreover $H=\Spec(C)$ and $K=\Spec(D)$ with $A \ra C,D$ projective and $B \twoheadrightarrow C,D$. Let $I=\Ker(B \ra C)$ and $J=\Ker(B \ra D)$; we want $J\sub I$. $J$ and $I$ are also projective as $A$-modules, so localising further on $S$ we may assume that $I,J,C,D$ are all free over $A$ (which implies that $B$ is free as well, since $B\cong I\oplus C \cong J \oplus D$). Choose a basis $e_{1},\dots,e_{r},\dots,e_{t}$ of $B$ over $A$ such that $e_{1},\dots,e_{r}$ is a basis of $I$, and choose another basis $f_{1},\dots,f_{s},\dots,f_{t}$ of $B$ over $A$ such that $f_{1},\dots,f_{s}$ is a basis for $J$. We can write
$$ f_{j}=\sum_{i=1}^{t}a_{ji}e_{i} $$
for unique $a_{ji}\in A$. To check that $J\sub I$ we need to check that $a_{ji}=0$ when $1\leq j \leq s$ and $i>r$. But this can be checked at geometric points of $S$ since $S$ is reduced. 

\medskip
So, let us go back to our original situation. Let $x : \Spec(\ka) \ra S$ be a geometric point.  We need to show that $ \varpi : \G_{x} \ra \G_{x}$ factors through $Fr_{q} : \G_{x} \ra \G^{(q)}_{x}$. This follows from a direct calculation on the Dieudonn\'e module. In fact, if $h$ is the height of $\G_{x}^{et}$, then $F^{f(n-h)}$ acts as $\varpi\sigma^{f(n-h)}$ on the Dieudonn\'e module of $\G_{x}^{0}$ and as $\sigma^{f(n-h)}$ on the Dieudonn\'e module of $\G_{x}^{et}$ by the description of the Dieudonn\'e modules above; this implies what we want. 
\end{proof}

\begin{remark}\label{strata}
The proof above works to give `strata' Hasse invariants cutting out the Ekedahl--Oort strata, in the sense of \cite{box,gk}. These strata Hasse invariants were already defined by Ito \cite{ito1,ito2}. More precisely, assume that there are no points $s$ of $S$ where $\G^{et}_{x}$ has height $>h$. Then the proof above shows that there exists an isogeny $\ol{V}_{h} : \G^{(q^{n-h})} \ra \G$ such that $\ol{V}_{h} \circ Fr_{q^{n-h}}= \vp$, and $\ol{V}_{h}^{\ast} : \omega_{\G} \ra \omega_{\G}^{q^{n-h}}$ defines a section $H_{h}\in H^{0}(S, \omega_{\G}^{q^{n-h}-1})$. Moreover, the proof of Proposition \ref{Hasse basic} adapts easily to show that the non-vanishing locus of $H_h$ is precisely the open subset consisting of the points $s$ where $\G_{s}^{et}$ has height $h$.

\medskip
In the context of Harris--Taylor Shimura varieties, this gives sections defined on the closure of each $\Xb_{\ul{m}}^{(h)}$ whose vanishing locus is precisely $\Xb_{\ul{m}}^{(h)}$ (we recall that the stratification given by the $\Xb_{\ul{m}}^{(h)}$ is precisely the Ekedahl--Ort stratification in this case, moreover it is also equal to the Newton stratification). This was the main point of Ito's work, and some further properties and applications are stated in \cite{ito1} in the case when $F^{+} = \Q$. 
\end{remark}

Moving on, we record some basic properties of our Hasse invariants.

\begin{proposition}\label{Hasse basic}
Let $S/k$ be a scheme and let $\G/S$ be a one-dimensional compatible Barsotti--Tate $\ok$-module of height $n$. Assume that the $\mu$-ordinary Hasse invariant of $\G$ exists and denote it by $H\in H^{0}(S,\omega_{\G}^{q-1})$.
\begin{enumerate}
\item Let $\phi : S^{\prime} \ra S$ be a $k$-morphism and let $\G^{\prime}=\G\times_{S}S^{\prime}$. Then the $\mu$-ordinary Hasse invariant of $\G^{\prime}$ exists and is equal to $\phi^{\ast}H$.

\item Assume that $S=\Spec \ka$, where $\ka$ is an algebraically closed field. Then $H\neq 0$ if and only if $\G$ is $\mu$-ordinary.
\end{enumerate}
\end{proposition}

\begin{proof}
The first part follows from the fact that both $Fr_{q}$ and $\varpi$ are functorial, so the factorization $\ol{V} \circ Fr_q = \varpi$ on $S$ pulls back to a factorization $\phi^{\ast}\ol{V} \circ Fr_q = \varpi$ on $S^{\prime}$.

\medskip
For the second part, we note that $H\neq 0$ if and only if $\ol{V}$ is \'etale. Let $h$ denote the height of $\G^{et}$; by the calculation in the proof of Proposition \ref{factoring} we see that $\varpi$ factors through $Fr_{q^{(n-h)}}$ so we must have $h=n-1$ for $\ol{V}$ to be \'etale. The calculation also shows that if $h=n-1$ then $\ol{V}$ is \'etale, which is what we wanted.
\end{proof}

In particular, we have a $\mu$-ordinary Hasse invariant whenever $\G/S$ comes by pullback from some $\G^{\prime}/S^{\prime}$ with $S^{\prime}$ reduced, and the non-vanishing locus is precisely the open whose geometric points $x$ are those for which $\G_{x}$ is $\mu$-ordinary.

\begin{remark}\label{Hasse invariants and quotient by kernel of Frobenius}
We note a particular consequence of Proposition \ref{Hasse basic}(1). Let $\G/S$ be a one-dimensional Barsotti--Tate $\ok$-module of height $n$ over a $k$-scheme $S$, and assume that the $\mu$-ordinary Hasse invariant $H(\G)$ exists. Let $m\geq 1$ and consider the $q^{m}$-power Frobenius twist $\G^{(q^m)}$, which is the pullback of $\G$ under the absolute $q^m$-th power Frobenius map $F_{q^m} : S \ra S$. Then Proposition \ref{Hasse basic}(1) implies that $H(\G^{(q^m)}) = F_{q^m}^* H(\G) = H(\G)^{q^m}$. Note that the $q^m$-power Frobenius isogeny $Fr_{q^m} : \G \ra \G^{(q^m)}$ gives a canonical isomorphism $\G/\Ker Fr_{q^m} \cong \G^{(q^m)}$, so we get that $H(\G/\Ker Fr_{q^m}) = H(\G)^{q^m}$.
\end{remark}

\medskip
Let us now return to the setting of our Shimura varieties. Recall $\Xb_{\ul{m}}$, which is reduced and has the one-dimensional compatible Barsotti--Tate $\ok$-module $\G$ on it, so we have a $\mu$-ordinary Hasse invariant $H\in H^{0}(\Xb_{\ul{m}},\omega_{\G}^{q-1})$. The $\mu$-ordinary locus is $\Xb_{\ul{m}}^{(n-1)}$. The following proposition is presumably well known to experts. We state it for completeness and sketch the proof, though it is not necessary for the main results of this paper.

\begin{proposition}
In the setting above, $\omega_{\G}$ is ample. As a consequence, $\Xb_{\ul{m}}^{(n-1)}$ is affine.
\end{proposition}

\begin{proof}
When $p$ is unramified in $F$ and $\ul{m}=(0,\dots,0)$ this is a special case of \cite[Proposition 7.15]{ls}, but the proof of that result also works when $p$ is ramified in $F$, using that the models $\Xf_{\ul{m}}$ are smooth and defined by a Kottwitz condition when $m_{1}=0$. The case of general $\ul{m}$ then follows since the natural map $\Xf_{\ul{m}} \ra \Xf_{(0,\dots,0)}$ is finite and surjective.
\end{proof}

\begin{remark}
By Remark \ref{strata}, it follows more generally that $\Xb_{\ul{m}}^{(h)}$ is affine for all $0\leq h \leq n-1$.
\end{remark}

\subsection{Canonical subgroups}\label{canonicalsubgroupssection}

Our goal in this section is to establish a theory of canonical subgroups for one-dimensional Barsotti--Tate $\ok$-modules of height $n$, under the assumption that the Hasse invariant exists. We follow the approach of Scholze closely \cite[3.2.1]{tor}, which relies on Illusie's deformation theory for group schemes \cite{illgp}.

\medskip
Let $\Qp^{cycl}$ denote the completion of the $p$-power cyclotomic extension of $\Qp$; this is a perfectoid field. We let $\Zp^{cycl}$ denote the ring of integers of $\Qp^{cycl}$. Set $K^{cycl} := K.\Qp^{cycl}$ and $\ok^{cycl}:=\oo_{K^{cycl}}$. Let $e^{\prime}_{n}:=\gcd(e,(p-1)p^{n})$, where we recall that $e$ is the ramification index of $K/\Qp$. Let $e^{\prime}=\lim_{n\ra \infty} e_{n}^{\prime}$, which exists since $(e_{n}^{\prime})_{n}$ is eventually constant. Then $\ok^{cycl}$ contains elements of valuation $\epsilon$ for any $\epsilon\in \Q_{\geq 0}$ of the form $ae^{\prime}/(p-1)p^{n}$ for $a,n\in \Z_{\geq 0}$ (here we normalise the valuation so that $\vp$ has valuation $1$); we will let $\vp^{\epsilon}$ denote such an element.

\medskip
The following results are direct analogues of \cite[Corollary 3.2.2, Corollary 3.2.6]{tor}.

\begin{proposition}\label{deform}
Let $R$ be a $\vp$-adically complete flat $\ok^{cycl}$-algebra. Let $G$ be a finite locally free commutative group scheme over $R$ and let $C_{1}\sub G_{1}:=G \otimes_{R}R/\vp$ be a finite locally free subgroup scheme. Assume that multiplication by $\vp^{\epsilon}$ on the Lie complex $\check{\ell}_{G_{1}/C_{1}}$ of $G_{1}/C_{1}$ is homotopic to zero, where $0\leq \epsilon < 1/2$. Then there is a finite locally free subgroup scheme $C\sub G$ such that $C\otimes_{R}R/\vp^{1-\epsilon} = C_{1}\otimes_{R/\vp}R/\vp^{1-\epsilon}$.  
\end{proposition}

\begin{proof}
The proof of \cite[Corollary 3.2.2]{tor} goes through verbatim (substituting $\vp$ for $p$).
\end{proof}

\begin{proposition}\label{canonicalsubgroup}
Let $R$ be a $\vp$-adically complete flat $\ok^{cycl}$-algebra and let $\G$ be a one-dimensional compatible Barsotti--Tate $\ok$-module of height $n$ over $R$, with reduction $\G_{1}$ to $R/\vp$. Assume that the $\mu$-ordinary Hasse invariant $H(\G_{1})$ exists and that $H(\G_{1})^{\frac{q^{m}-1}{q-1}}$ divides $\vp^{\epsilon}$ for some $\epsilon < 1/2$. Then there is a unique finite locally free subgroup scheme $C_{m}\sub \G[\vp^{m}]$ such that $C_{m}\otimes_{R}R/\vp^{1-\epsilon}=(\Ker Fr_{q^m})\otimes_{R/\vp}R/\vp^{1-\epsilon}$.

For any $\vp$-adically complete flat $\ok^{cycl}$-algebra $R^{\prime}$ with an $\ok^{cycl}$-algebra map $R \ra R^{\prime}$ , one has
\begin{equation}\label{points}
C_{m}(R^{\prime})=\{ s \in \G[\vp^{m}](R^{\prime}) \mid s \equiv 0 \mod \vp^{(1-\epsilon)/q^{m}} \}.
\end{equation}
\end{proposition}

\begin{proof}
The proof of \cite[Corollary 3.2.6]{tor} goes through with only superficial changes; we sketch it for completeness. Fix $m$ and set $H_{1}:=\Ker (\ol{V}^{m} : \G_{1}^{(q^{m})} \ra \G_{1})$ (which makes sense by assumption); then there is an exact sequence
$$ 0 \ra \Ker Fr_{q^m} \ra \G_{1}[\vp^{m}] \ra H_{1} \ra 0 $$
by definition. By Lemma \ref{liecomplex} below, the Lie complex of $H_{1}$ is isomorphic to
$$ \check{\ell}_{H_{1}}=(\Lie \G_{1}^{(q^{m})} \overset{\Lie \ol{V}^m}{\longrightarrow} \Lie \G_{1}). $$
We calculate the determinant of $\Lie \ol{V}^m$ to be $H(\G_{1})^{\frac{q^{m}-1}{q-1}}$ using Remark \ref{Hasse invariants and quotient by kernel of Frobenius}. Multiplication by the determinant $\Lie \ol{V}^m$ is then null-homotopic on the complex $\Lie \G_{1}^{(q^{m})} \overset{\Lie \ol{V}^m}{\longrightarrow} \Lie \G_{1}$ (using the adjugate endomorphism of $\Lie \ol{V}^m$ as the chain homotopy), so multiplication by $\vp^{\epsilon}$ is null-homotopic using the assumption that $H(\G_{1})^{\frac{q^{m}-1}{q-1}}$ divides $\vp^{\epsilon}$. The existence of $C_{m}$ then follows from Proposition \ref{deform}. Uniqueness is a consequence of the final statement of the proposition, which is proved in the same way as the analogous part of \cite[Corollary 3.2.6]{tor}, using Lemma \ref{def}.
\end{proof}

We have used the following two lemmas in the proof.

\begin{lemma}\label{def}
Let $R$ be a $\vp$-adically complete flat $\ok^{cycl}$-algebra. Let $X/R$ be an affine scheme such that $\Omega_{X/R}^{1}$ is killed by $\vp^\epsilon$, for some $\epsilon \geq 0$. Let $s,t\in X(R)$ be two sections with $\ol{s}=\ol{t} \in X(R/\vp^{\delta})$, for some $\delta > \epsilon$. Then $s=t$.
\end{lemma}

\begin{proof}
The proof of \cite[Lemma 3.2.4]{tor} goes through, replacing $p^\epsilon$ and $p^\delta$ by $\vp^\epsilon$ and $\vp^\delta$, respectively.
\end{proof}

\begin{lemma}\label{liecomplex}
With notation as in the statement and proof of Proposition \ref{canonicalsubgroup}, the Lie complex $\check{\ell}_{H_{1}}$ of $H_{1}$ is isomorphic to the complex $ \Lie \G_{1}^{(q^{m})} \overset{\Lie \ol{V}^m}{\longrightarrow} \Lie \G_{1} $ (with terms in degrees $0$ and $1$).
\end{lemma}

\begin{proof}
We may identify $\Lie \G_1$ and $\Lie \G_1^{(q^m)}$ with $\Lie \G_1[\vp^k]$ and $\Lie \G_1^{(q^m)}[\vp^k]$, respectively, for all large enough $k$. Note that we have natural identifications $\Lie \G_1[\vp^k] = \check{\ell}_{\G_1[\vp^k]}^{\leq 0}$ and $\Lie \G_1^{(q^m)}[\vp^k]= \check{\ell}_{\G_1^{(q^m)}[\vp^k]}^{\leq 0}$ (cf. e.g. \cite[\S 2.1]{ill}; we regard modules as complexes concentrated in degree $0$). We have exact sequences
$$ 0 \ra H_1 \ra \G_1^{(q^m)}[\vp^k] \ra \G_1^{(q^m)}[\vp^k]/H_1 \ra 0 $$
for all large $k$, which give distinguished triangles
$$ \check{\ell}_{H_1} \ra \check{\ell}_{\G_1^{(q^m)}[\vp^k]} \ra \check{\ell}_{\G_1^{(q^m)}[\vp^k]/H_1} \ra .$$
Define $A$ to be the complex $ \Lie \G_{1}^{(q^{m})} \overset{\Lie \ol{V}^m}{\longrightarrow} \Lie \G_{1} $. By the remarks above, we have 
$$ A = {\rm cone}\left( \check{\ell}_{\G_1^{(q^m)}[\vp^k]}^{\leq 0} \ra \check{\ell}_{\G_1[\vp^k]}^{\leq 0} \right) [-1] $$
and hence a distinguished triangle $A \ra \check{\ell}_{\G_1^{(q^m)}[\vp^k]}^{\leq 0} \ra \check{\ell}_{\G_1[\vp^k]}^{\leq 0} \ra $. We may then construct a commutative diagram
$$
\xymatrix{A \ar[r]\ar[d]^{f} & \check{\ell}_{\G_1^{(q^m)}[\vp^k]}^{\leq 0} \ar[r]\ar[d] & \check{\ell}_{\G_1[\vp^k]}^{\leq 0} \ar[r]\ar[d] &  \\
\check{\ell}_{H_1} \ar[r] & \check{\ell}_{\G_1^{(q^m)}[\vp^{k^{\prime}}]} \ar[r] &  \check{\ell}_{\G_1^{(q^m)}[\vp^{k^{\prime}}]/H_1} \ar[r] & }
$$
for all large enough $k^{\prime}\geq k$, where the two unmarked vertical arrows are canonical and $f$ then exists for abstract reasons (we remark that we can and do choose $f$ to be independent of $k^\prime$). We claim that $f$ is an isomorphism; it suffices to check this on cohomology groups in degrees $0$ and $1$ (all other cohomology groups vanish). Taking long exact exact sequences we get a commutative diagram (with exact rows)
$$
\xymatrix{0 \ar[r] & H^0(A) \ar[r]\ar[d]^{H^0(f)} & \Lie \G_1^{(q^m)}[\vp^k] \ar[r]\ar[d] & \Lie \G_1[\vp^k] \ar[r]\ar[d] &  H^1(A) \ar[r]\ar[d]^{H^{1}(f)} & 0 \ar[d] \\
0 \ar[r] & H^0(\check{\ell}_{H_1}) \ar[r] & \Lie \G_1^{(q^m)}[\vp^{k^{\prime}}] \ar[r] &  \Lie \G_1^{(q^m)}[\vp^{k^{\prime}}]/H_1 \ar[r] & H^1(\check{\ell}_{H_1}) \ar[r] & H^1\left(\check{\ell}_{\G_1^{(q^m)}[\vp^{k^{\prime}}]}\right).  }
$$
Now take the direct limit over $k^{\prime}$ in the bottom row. We have $\varinjlim_{k^{\prime}} H^1\left(\check{\ell}_{\G_1^{(q^m)}[\vp^{k^{\prime}}]}\right) =0 $ by \cite[Proposition 2.2.1(c)(i)]{ill}, and the maps $\Lie \G_1^{(q^m)}[\vp^k] \ra \varinjlim_{k^\prime} \Lie \G_1^{(q^m)}[\vp^{k^{\prime}}]$ and 
$$ \Lie \G_1[\vp^k] \ra \varinjlim_{k^\prime} \left( \Lie \G_1^{(q^m)}[\vp^{k^{\prime}}]/H_1 \right) \cong \varinjlim_{k^{\prime\prime}} \Lie \G_1[\vp^{k^{\prime\prime}}] $$
are both isomorphisms. This implies that $H^0(f)$ and $H^1(f)$ are both isomorphisms, which finishes the proof.
\end{proof}

\begin{remark}
Morally, the Lemma above should be proven by taking the homotopy colimit of the triangles $ \check{\ell}_{H_1} \ra \check{\ell}_{G_1^{(q^m)}[\vp^k]} \ra \check{\ell}_{G_1^{(q^m)}[\vp^k]/H_1} \ra $ for large $k$. However, since homotopy colimits are poorly behaved, such an argument seems to require some work to carry out. The argument above may be viewed as an elementary workaround.
\end{remark}

Using Proposition \ref{canonicalsubgroup}, we define canonical subgroups by analogy with \cite[Definition 3.2.7]{tor}.

\begin{definition}
Let $R$ be a $\vp$-adically complete flat $\ok^{cycl}$-algebra and let $\G$ be a one-dimensional compatible Barsotti--Tate $\ok$-module of height $n$ over $R$, with reduction $\G_{1}$ to $R/\vp$. We say that $\G$ has a weak canonical subgroup of level $m$ if the $\mu$-ordinary Hasse invariant $H(\G_{1})$ exists and $H(\G_{1})^{\frac{q^{m}-1}{q-1}}$ divides $\vp^{\epsilon}$ for some $\epsilon < 1/2$, and we then call the subgroup $C_{m}\sub \G[\vp^{m}]$ (given by Proposition \ref{canonicalsubgroup}) the weak canonical subgroup of level $m$. If in addition $H(\G_{1})^{q^{m}}$ divides $\vp^{\epsilon}$, we call $C_{m}$ the (strong) canonical subgroup.
\end{definition}

One then has the following analogue of \cite[Proposition 3.2.8]{tor}, which is proved by exactly the same arguments.

\begin{proposition}\label{properties}
Let $R$ be a $\vp$-adically complete flat $\ok^{cycl}$-algebra, and let $\G$ and $\mc{H}$ be one-dimensional compatible Barsotti--Tate $\ok$-modules of height $n$ over $R$.
\begin{enumerate}
\item If $\G$ has a (weak) canonical subgroup of level $m$, then it has a (weak) canonical subgroup of level $m^{\prime}$ for any $m^{\prime}\leq m$, and $C_{m^{\prime}}\sub C_{m}$.

\item Let $f : \G \ra \mc{H}$ be a morphism of Barsotti--Tate $\ok$-modules. If both $\G$ and $\mc{H}$ have canonical subgroups $C_{m}$ and $D_{m}$, respectively, of level $m$, then $f$ maps $C_{m}$ into $D_{m}$. In particular, $C_m$ is stable under the action of $\ok$.

\item Assume that $\G$ has a canonical subgroup $C_{m_{1}}$ of level $m_{1}$, and that $\mc{H}=\G/C_{m_{1}}$. Then $\mc{H}$ has a canonical subgroup $D_{m_{2}}$ of level $m_{2}$ if and only if $\G$ has a canonical subgroup $C_{m_{1}+m_{2}}$ of level $m_{1}+m_{2}$. If so, there is a short exact sequence
$$ 0 \ra C_{m_{1}} \ra C_{m_{1}+m_{2}} \ra D_{m_{2}} \ra 0 $$
which is compatible with $0 \ra C_{m_{1}} \ra \G \ra \mc{H} \ra 0 $.

\item Assume that $\G$ has a canonical subgroup $C_{m}$ of level $m$ and let $x$ be a geometric point of $\Spec R[\vp^{-1}]$. Then $C_{m}(x)\cong \ok/\vp^{m}$ as $\ok$-modules. In other words, the restriction of $\G$ to $\Spec R[\vp^{-1}]$ is \'etale-locally isomorphic to $\ok/\vp^m$ as a finite \'etale group scheme with an $\ok$-action.
\end{enumerate}
\end{proposition}

\section{Perfectoid Shimura varieties}\label{sec3}

In this section we prove our results about Harris--Taylor Shimura varieties. We first prove an analogue of Scholze's result \cite[Corollary 3.2.19]{tor} that the `anticanonical tower' for Siegel modular varieties is perfectoid at $\Gamma_{0}(p^{\infty})$-level; this is the main result of this section. Using this, we prove slight refinements\footnote{The (very minor) refinement is the following: \cite{tor,cs} work over the full infinite level at all places dividing $p$, whereas we only work with full infinite level at the place $v$.} of results of Scholze \cite{tor} and Caraiani--Scholze \cite{cs} that the tower of Harris--Taylor Shimura varieties is perfectoid at full infinite level and admits a Hodge--Tate period map to $\mb{P}^{n-1}$. For this, we follow Scholze's arguments for the Siegel case, but the situation is much simpler in our case due to the absence of a boundary. We also take advantage of the formalism of diamonds, which provide a good setting in which to carry out the arguments.

\subsection{The anticanonical tower}

Let us start by recalling a characteristic $0$ version of the moduli problem defining our Shimura varieties from \cite[\S III.1]{ht}. For each $i \in \{1,\dots,r\}$, let
$$ U_{v_{i}} \sub (\oo_{B, w_{i}}^{op})^{\times} $$
be a compact open subgroup and set
$$ U_{p} = \Zp^{\times} \times \prod_{i=1}^{r}U_{v_{i}} \sub G(\Qp) $$
and $U=U^{p}U_{p}$ (recall that we have fixed a sufficiently small compact open subgroup $U^{p}\sub G(\mb{A}^{p,\infty})$ throughout this article). We define a contravariant functor $X_{U}$ from locally noetherian $K$-schemes to sets as follows. If $S$ is a connected locally Noetherian $K$-scheme and $s$ is a geometric point of $S$, we define $X_{U}(S,s)$ to be the set of equivalence classes of $(r+4)$-tuples $(A,\lambda,i, \ol{\eta}^{p},\ol{\eta}_{i})$ where
\begin{itemize}
\item $A$ is an abelian scheme over $S$ of dimension $dn^2$;

\item $\lambda : A \ra A^{\vee}$ is a polarization;

\item $i : B \hookrightarrow \End_{S}(A)\otimes_{\Z}\Q$ is a homomorphism such that $(A,i)$ is compatible and $\lambda \circ i(b) = i(b^{\ast})^{\vee}\circ \lambda$ for all $b\in B$;

\item $\ol{\eta}^{p}$ is a $\pi_{1}(S,s)$-invariant $U^{p}$-orbit of isomorphisms of $B\otimes_{\Q} \A^{p,\infty}$-modules $\eta : V\otimes_{\Q}\A^{p,\infty} \ra V^{p}A_{s}$ which take the standard pairing $(-,-)$ on $V$ to a $(\A^{p,\infty})^{\times}$-multiple of the $\lambda$-Weil pairing on $V^{p}A_{s}$;

\item $\ol{\eta}_{1}$ is $\pi_{1}(S,s)$-invariant $U_{v_{1}}$-orbit of isomorphisms $\eta_{1} : \Lambda_{11}\otimes_{\Zp}\Qp \ra \epsilon V_{w_{1}}A_{s}$ of $K$-modules;

\item for $i=2,\dots,r$, $\ol{\eta}_{i}$ is a $\pi_{1}(S,s)$-invariant $U_{v_{i}}$-orbit of isomorphisms of $B_{w_{i}}$-modules $\eta_{i} : \Lambda_{i} \otimes_{\Zp}\Qp \ra V_{w_{i}}A_{s}$.
\end{itemize}

Before defining equivalence, let us define compatibility. The map $i$ induces an action of $E$ on $\Lie A$, and we let $\Lie^{+}A$ denote the summand of $\Lie A$ where $E$ acts in the same way as via the structure morphism $E \ra \oo_{S}$. We then say that $(A,i)$ is compatible if $\Lie^{+}A$ has rank $n$ (over $\oo_{S}$) and the actions of $F^{+}$ on $\Lie^{+}A$ via $i$ and via the structure morphism $F^{+} \ra \oo_{S}$ agree. Finally, two $(r+4)$-tuples $(A,\lambda,i,\ol{\eta}^p,\ol{\eta}_i)$ and $(A^{\prime},\lambda^{\prime},i^{\prime},{\ol{\eta}^{\prime}}^p,\ol{\eta}^{\prime}_i)$ are equivalent if there is an isogeny $\alpha : A \ra A^{\prime}$ which takes $\lambda$ to a $\Q^{\times}$-multiple of $\lambda^{\prime}$, takes $i$ to $i^{\prime}$ and takes $\ol{\eta}$ to $\ol{\eta}^{\prime}$. Again the set $X_{U}(S,s)$ is canonically independent of the choice of $s$, giving $X_{U}$ on connected $S$, and one extends to disconnected $S$ in the usual way. This functor is representable by a smooth projective $K$-scheme which we will also denote by $X_{U}$. If $\ul{m}=(m_{1},\dots,m_{r})$ and $U_{v_{i}}=1+\vp_{i}^{m_{i}}\oo_{B,w_{i}}^{op}$, then $X_{U}$ is canonically isomorphic to the generic fibre $X_{\ul{m}}$ of $\Xf_{\ul{m}}$ ; see \cite[pp. 93-94]{ht}.

\medskip
For the rest of this article, we will fix non-negative integers $m_{2},\dots,m_{r}$ and the corresponding compact open subgroups $U_{v_{i}}=1+\vp_{i}^{m_{i}}\oo_{B,w_{i}}^{op}$ for $i
=2,\dots,r$. We drop the levels $U^{p}$, $U_{v_{i}}$, $i=2,\dots,r$, and $\Zp^{\times}$ from all notation and only indicate the level at $v$. In particular, we write $X_{m}$ for what was previously called $X_{(m,m_{2},\dots,m_{r})}$, etc. 

\medskip
Let us now introduce the level subgroups $U_{0}(\vp^{m})\sub \GL_{n}(K)$ that we will use to define the anticanonical tower. Let $P\sub \GL_{n}$ denote the $(n-1,1)$-block upper triangular parabolic. We define, for $m \geq 0$,
$$ U_{0}(\vp^{m}) := \{ g\in \GL_{n}(\ok) \mid g\,\, {\rm mod}\,\, \vp^{m} \in P(\ok/\vp^{m}) \}. $$
Let us also put $U(\vp^{m})=1+\vp^{m}M_{n}(\ok)$. Consider $X_{U_{0}(\vp^{m})}$. It is the quotient of $X_{m}$ by the free action of the finite group $U_{0}(\vp^{m})/U(\vp^{m})\cong P(\ok/\vp^{m})$. Since the level structure at $w$ defining $X_{m}$ are isomorphisms
$$ \alpha_{1} : \vp^{-m}\Lambda_{11}/\Lambda_{11} \ra \G[\vp^{m}], $$
it follows that the level structure at $w$ defining $X_{U_{0}(\vp^{m})}$ are $\ok$-subgroup schemes $H\sub \G[\vp^{m}]$ which are \'etale-locally isomorphic to $(\ok/\vp^{m})^{n-1}$.

\medskip

For the rest of this section we will base change all Shimura varietes $X_U$ to $K^{cycl}$.  
We will now define some formal schemes whose generic fibres embed in the rigid analytification of $X_{U_{0}(\vp^{m})}$ (for suitable $m$). Set $\Xf:=\Xf_{0}$ and let $\wh{\Xf}$ be the formal completion of $\Xf \otimes_{\ok}\ok^{cycl}$ along $\vp$. Recall our conventions about elements $\epsilon\in \Q_{\geq 0}$ and elements $\vp^{\epsilon}\in \ok^{cycl}$ from \S \ref{canonicalsubgroupssection}.

\begin{definition}
Assume that $0\leq \epsilon < 1/2$. Let $\wh{\Xf}(\epsilon) \ra \wh{\Xf}$ be the functor on $\vp$-adically complete flat $\ok^{cycl}$-algebras sending such an $S$ to the set of equivalence classes of pairs $(f,u)$, where $f : \Spf S \ra \wh{\Xf}$ is a morphism and and $u\in H^{0}(\Spf S, (f^{\ast}\omega)^{1-q})$ is a section such that $u(f^{\ast}H)=\vp^{\epsilon}\in S/\vp$, where $H$ is the $\mu$-ordinary Hasse invariant on $\wh{\Xf}\otimes_{\ok^{cycl}}\ok^{cycl}/\vp$. Two pairs $(f,u)$ and $(f^{\prime},u^{\prime})$ are equivalent if $f=f^{\prime}$ and there is some $h\in S$ with $u^{\prime}=u(1+\vp^{1-\epsilon}h)$.
\end{definition} 

\begin{proposition}\label{epsilonnhood def}
$\wh{\Xf}(\epsilon)$ is representable by a flat formal scheme over $\ok^{cycl}$ which is affine over $\wh{\Xf}$. 
\end{proposition}

\begin{proof}
It suffices to work Zariski locally on $\wh{\Xf}$, so let $\Spf R \sub \wh{\Xf}$ be an affine open over which $\omega^{q-1}$ is trivial. Choose a  non-vanishing section $\eta\in \omega^{q-1}$ and choose a lift $\wt{H}\in H^{0}(\Spf R, \omega^{q-1})$ of $H$. We claim that $\wh{\Xf}(\epsilon)\times_{\wh{\Xf}}\Spf R$ is represented by $\Spf (R\langle T \rangle / (T(\wt{H}\eta^{-1})-\vp^{\epsilon}))$. The formal scheme $\Spf (R\langle T \rangle / (T(\wt{H}\eta^{-1})-\vp^{\epsilon}))$ represents pairs $(f,\wt{u})$ with $f : \Spf S \ra \Spf R$ a morphism and $\wt{u}\in H^{0}(\Spf S, (f^{\ast}\omega)^{1-q})$ such that $\wt{u}\wt{H}=\vp^{\epsilon}$ in $S$. There is a natural transformation from pairs $(f,\wt{u})$ to equivalence classes of pairs $(f,u)$ parametrized by $\wt{\Xf}(\epsilon)\times_{\wt{\Xf}}\Spf R$, and one shows that this is an isomorphism by the same argument as in \cite[Lemma 3.2.13]{tor}. This shows that $\wh{\Xf}(\epsilon)$ is representable and is affine over $\wh{\Xf}$.

\medskip
It remains to show that $R\langle T \rangle / (T(\wt{H}\eta^{-1})-\vp^{\epsilon})$ is flat over $\ok^{cycl}$, for which it suffices to show that it has no $\vp^{\epsilon}$-torsion. Set $A=R\langle T \rangle$ and $g=T(\wt{H}\eta^{-1})-\vp^{\epsilon}$. Taking the long exact sequence of $0 \ra A \ra A \ra A/g \ra 0$ and using the $\ok^{cycl}$-flatness of $A$ shows that $\Tor_{1}^{\ok^{cycl}}(\ok^{cycl}/\vp^{\epsilon},A/g)$ (which is the $\vp^{\epsilon}$-torsion in $A/g$) is the $g$-torsion in $A/\vp^{\epsilon}$. Since $g=T(H\eta^{-1})$ in $A/\vp^{\epsilon}$ and $H\eta^{-1}$ is not a zero divisor in $R/\vp^{\epsilon}$, the assertion follows.
\end{proof}

For any formal scheme whose notation involves $\wh{\Xf}...$, we will use $\mc{X}...$ to denote its generic fibre, and $\Xb ...$ the reduction modulo $\vp$. We record two corollaries.

\begin{corollary}\label{mod p}
The reduction $\Xb(\epsilon)$ represents the functor on $\ok^{cycl}/\vp$-algebras sending such an $S$ to the set of pairs $f: \Spec S \ra \Xb$ and $u\in H^{0}(\Spec S, (f^{\ast}\omega)^{1-q})$ such that $u(f^\ast H)=\vp^{\epsilon}$.
\end{corollary}

\begin{proof}
It suffices to prove this locally on $\Xb$, so we pick an open affine $\Spf R \sub \wh{\Xf}$ and $\eta$ trivialising $\omega^{q-1}$ as in the proof of Proposition \ref{epsilonnhood def}. Then, by the proof, $\Xb(\epsilon)$ is represented over $\Spec R/\vp$ by the $\ok^{cycl}/\vp$-algebra $(R/\vp)[T]/(T(H\ol{\eta}^{-1})-\vp^{\epsilon})$, where $\ol{\eta}$ denotes the reduction of $\eta$. A morphism $(R/\vp)[T]/(T(H\ol{\eta}^{-1})-\vp^{\epsilon}) \ra S$ then corresponds to a morphism $R/\vp \ra S$ plus an element $t\in S$ such that $t(H\ol{\eta}^{-1})=\vp^{\epsilon}$; setting $u=t\ol{\eta}^{-1}$ gives the desired element of $H^{0}(\Spec S, (f^{\ast}\omega)^{1-q})$. One checks that this is independent of the choice of $\eta$, which finishes the proof.
\end{proof}

\begin{corollary}\label{lifting}
Let $0\leq \epsilon <1/2$. Let $S$ be a $\vp$-adically complete and flat $\ok^{cycl}$-algebra and let $f :\Spf S \ra \wh{\Xf}$ be a morphism. Assume that the reduction $\ol{f} : \Spec S/\vp^{1-\epsilon} \ra \Xb \otimes_{\ok^{cycl}/\vp} \ok^{cycl}/\vp^{1-\epsilon}$ lifts to a map $\ol{g} : \Spec S/\vp^{1-\epsilon} \ra \Xb(\epsilon) \otimes_{\ok^{cycl}/\vp} \ok^{cycl}/\vp^{1-\epsilon}$. Then there exists a unique map $g : \Spf S \ra \wh{\Xf}(\epsilon)$ lifting $\ol{g}$ such that the composition $ \Spf S \overset{g}{\longrightarrow} \wh{\Xf}(\epsilon) \ra \wh{\Xf} $ is $f$.
\end{corollary}

\begin{proof}
The assertion is local on the target and the source, so we may use the local description of $\wh{\Xf}(\epsilon)$ from the proof of Proposition \ref{epsilonnhood def}; we use the notation of that proof. The problem then becomes to prove the following: If $h : R \ra S$ is an $\ok^{cycl}$-algebra homomorphism and $u_{0}\in S$ is an element such that $u_{0}h(\wt{H}\eta^{-1}) \equiv \vp^{\epsilon}$ modulo $\vp^{1-\epsilon}$, then there is a unique $u\in S$ such that $uh(\wt{H}\eta^{-1})=\vp^{\epsilon}$ and $u\equiv u_0$ modulo $\vp^{1-\epsilon}$. For existence, write $u_0 h(\wt{H}\eta^{-1}) = \vp^{\epsilon} + \vp^{1-\epsilon}v$ for some $v\in S$, then we can set $u=u_0 (1+\vp^{1-2\epsilon}v)^{-1}$. Since $S$ is $\ok^{cycl}$-flat, existence shows that $S$ is $h(\wt{H}\eta^{-1})$-torsionfree, which implies uniqueness.
\end{proof}

\begin{remark}\label{ordinary locus}
Note that the map $\wh{\Xf}(0) \ra \wh{\Xf}$ is an open immersion; it identifies $\wh{\Xf}(0)$ with the open subset $\{ H\neq 0 \}$ of $\wh{\Xf}$. In particular, $\wh{\Xf}(0)$ is formally smooth over $\ok^{cycl}$. Note also that, for any $0\leq \epsilon <1/2$, the natural map $\wh{\Xf}(0) \ra \wh{\Xf}(\epsilon)$ (given by multiplying the section by $\vp^{\epsilon}$) is an open immersion, again identifying $\wh{\Xf}(0)$ as the subset $\{ H\neq 0\} \sub \wh{\Xf}(\epsilon)$. Similar remarks then apply modulo $\vp$, in particular $\Xb(0)$ is formally smooth over $\ok^{cycl}/ \vp$.
\end{remark}

Let $\wh{\Af}$ be the universal abelian (formal) scheme over $\wh{\Xf}$, with pullback $\wh{\Af}(\epsilon)$ to $\wh{\Xf}(\epsilon)$. We may define canonical subgroups of $\wh{\Af}(\epsilon)$ whenever they exist for $\G_{\wh{\Af}(\epsilon)}$, as follows. Recall that we have a decomposition
$$ \wh{\Af}(\epsilon)[p^{\infty}] \cong \G_{\wh{\Af}(\epsilon)}^{\oplus n} \oplus \wh{\Af}(\epsilon)[w_{2}^{\infty}]\oplus \dots \oplus \wh{\Af}(\epsilon)[w_{r}^{\infty}] \oplus (\G_{\wh{\Af}(\epsilon)}^{\vee})^{\oplus n} \oplus \wh{\Af}(\epsilon)[w_{2}^{\infty}]^{\vee}\oplus \dots \oplus \wh{\Af}(\epsilon)[w_{r}^{\infty}]^{\vee}. $$
Here $-^{\vee}$ denotes the Cartier dual. If $\G_{\wh{\Af}(\epsilon)}$ has a (weak) canonical subgroup $C_{m}$ of level $m$, then we let $D_{m}\sub \wh{\Af}(\epsilon)[p^{m}]$ be the subgroup corresponding to 
$$ C_{m}^{\oplus n} \oplus 0 \oplus \dots \oplus 0 \oplus (C_{m}^{\perp})^{\oplus n} \oplus \wh{\Af}(\epsilon)[w_{2}^{m}]^{\vee} \oplus \dots \oplus \wh{\Af}(\epsilon)[w_{r}^{m}]^{\vee} $$
under the isomorphism above, where $C_{m}^{\perp}$ is the annihilator of $C_{m}$ with respect to the duality pairing. We say that $D_{m}$ is the (weak) canonical subgroup of $\wh{\Af}(\epsilon)$. Note that $D_{m}$ modulo $\vp$ is the kernel of the $q$th power Frobenius on $\Ab(\epsilon)$ (since $\wh{\Af}(\epsilon)[w_{i}^{\infty}]$ is \'etale for $i=2,\dots,r$).
\medskip

Next, we note that there is a natural isomorphism $\Xb^{(q)} \cong \Xb$ over $\ok^{cycl}/\vp$ (or any other base), since $\Xb$ comes by base change from $k$. Let $Fr=Fr_{\Xb/(\ok^{cycl}/\vp)} : \Xb \ra \Xb^{(q)}$ be the relative ($q$th power) Frobenius map\footnote{We apologise that the notation for Frobenius maps in this section differs slightly from the notation in section \ref{sec2}.}; note that the composition
$$ \Xb \overset{Fr}{\longrightarrow} \Xb^{(q)} \cong \Xb $$
is the map coming from the abelian scheme $\Ab/\Ker Fr_{\Ab/\Xb} \ra \Xb$ (with extra structures), where $Fr_{\Ab/\Xb}$ is the relative Frobenius. We may then pull back this situation to $\Xb(\epsilon)$ to obtain the following analogue of \cite[Lemma 3.2.14]{tor}.

\begin{lemma}\label{frob}
Let $0 \leq \epsilon <1/2$. The isomorphism $\Xb^{(q)} \cong \Xb$ induces an isomorphism $\Xb(q^{-1}\epsilon)^{(q)} \cong \Xb(\epsilon)$, and the composition $\Xb(q^{-1}\epsilon) \overset{Fr}{\longrightarrow} \Xb(q^{-1}\epsilon)^{(q)} \cong \Xb(\epsilon)$ is induced from the abelian scheme $\Ab(q^{-1}\epsilon)/\Ker Fr_{\Ab(q^{-1}\epsilon)/\Xb(q^{-1}\epsilon)} \ra \Xb(q^{-1}\epsilon)$ (with extra structures) together with the $q$-th power of the universal section on $\Xb(q^{-1}\epsilon)$.
\end{lemma}

\begin{proof}
That $\Xb^{(q)} \cong \Xb$ induces an isomorphism $\Xb(q^{-1}\epsilon)^{(q)} \cong \Xb(\epsilon)$ follows (for example) by explicit calculation in the local coordinates of the proof of Corollary \ref{mod p}, assuming in addition that the ring $R/\vp$ in that proof as well as the non-vanishing section $\ol{\eta}$ comes by base change from $k$. It then follows that $\Ab(\epsilon)$ pulls back to $\Ab(q^{-1}\epsilon)/\Ker Fr_{\Ab(q^{-1}\epsilon)/\Xb(q^{-1}\epsilon)}$ via the map $\Xb(q^{-1}\epsilon) \ra \Xb(\epsilon)$ since $\Ab$ pulls back to $\Ab/\Ker Fr_{\Ab/\Xb}$ via $Fr: \Xb \ra \Xb$ (with extra structures). Finally, one identifies the pullback of the universal section by explicit calculation in the local coordinates used in the first part of the proof.
\end{proof}

We will abuse the terminology and write $Fr$ for the map $\Xb(q^{-1}\epsilon) \ra \Xb(\epsilon)$, and refer to it as the relative Frobenius.

\begin{theorem}\label{anticanonical}
Let $0\leq \epsilon < 1/2$.
\begin{enumerate}
\item There is a unique morphism $\wt{F}: \wh{\Xf}(q^{-1}\epsilon) \ra \wh{\Xf}(\epsilon)$ which is equal to the relative Frobenius $\Xb(q^{-1}\epsilon) \ra \Xb(\epsilon)$ modulo $\vp^{1-\epsilon}$. $\wt{F}$ is finite, and its generic fibre is finite flat of degree $q^{n-1}$.

\smallskip
\item For any integer $m\geq 1$, the Barsotti--Tate $\ok$-module $\G_{\wh{\Af}(q^{-m}\epsilon)}$ admits a canonical subgroup $C_{m}$ of level $m$, and hence the abelian variety $\wh{\Af}(q^{-m}\epsilon)$ admits a canonical subgroup $D_{m}$ of level $m$. This induces an open immersion $\Xc(q^{-m}\epsilon) \ra \Xc_{U_{0}(\vp^{m})}$ given by the abelian variety $\Ac(q^{-m}\epsilon)/D_{m}$, the $\ok$-subgroup $\G_{\Ac(q^{-m}\epsilon)}[\vp^{m}]/C_{m}$, plus the induced extra structures. Moreover, the diagram
\begin{equation*}
\xymatrix{\Xc(q^{-m-1}\epsilon) \ar[r]\ar[d]^{\wt{F}} & \Xc_{U_{0}(\vp^{m+1})} \ar[d] \\
\Xc(q^{-m}\epsilon) \ar[r] & \Xc_{U_{0}(\vp^{m})}}
\end{equation*}
commutes and is cartesian.

\smallskip
\item There is a weak canonical subgroup $C\sub \G_{\wh{\Af}(\epsilon)}$ of level $1$. The open immersion $\Xc(q^{-1}\epsilon) \ra \Xc_{U_{0}(\vp)}$ identifies $\Xc(q^{-1}\epsilon)$ with the open subset $\Xc_{U_{0}(\vp)}(\epsilon)_{a}$ of $\Xc_{U_{0}(\vp)}$ where the Hasse invariant has valuation $\leq \epsilon$ and the $\ok$-subgroup $C^{\prime}\sub \G[\vp]$ satisfies $C\cap C^{\prime}=0$. 
\end{enumerate}
\end{theorem}

\begin{proof}
We start by proving (1). By Proposition \ref{canonicalsubgroup} there is a strong canonical subgroup $C$ of $\G_{\wh{\Af}(q^{-1}\epsilon)}$ (of level $1$), and hence a strong canonical subgroup $D$ of $\wh{\Af}(q^{-1}\epsilon)$. This gives an abelian variety $\wh{\Af}(q^{-1}\epsilon)/D \ra \wh{\Xf}(q^{-1}\epsilon)$ with extra structures, and hence a morphism $\wh{\Xf}(q^{-1}\epsilon) \ra \wh{\Xf}$. Note that $\wh{\Af}(q^{-1}\epsilon)/D \ra \wh{\Xf}(q^{-1}\epsilon)$ reduces to $\Ab(q^{-1}\epsilon)/\Ker Fr_{\Ab(q^{-1}\epsilon)/\Xb(q^{-1}\epsilon)} \ra \Xb(q^{-1}\epsilon)$ modulo $\vp^{1-\epsilon}$ by Proposition \ref{canonicalsubgroup}, so the map $\wh{\Xf}(\q^{-1}\epsilon) \ra \wh{\Xf}$ reduces to a map $\Xb(q^{-1}\epsilon) \ra \Xb$ modulo $\vp^{1-\epsilon}$ which lifts to the relative Frobenius $\Xb(q^{-1}\epsilon) \ra \Xb(\epsilon)$ modulo $\vp^{1-\epsilon}$ by Lemma \ref{frob}. Corollary \ref{lifting} then gives us a lift $\wt{F} : \wh{\Xf}(q^{-1}\epsilon) \ra \wh{\Xf}(\epsilon)$ of the relative Frobenius modulo $\vp^{1-\epsilon}$. The uniqueness follows from the uniqueness of the canonical subgroup (which establishes uniqueness of the lift $\wh{\Xf}(q^{-1}\epsilon) \ra \wh{\Xf}$) and the uniqueness part of Corollary \ref{lifting}.

\medskip
For finiteness, first note that the morphism is affine by construction. Finiteness of $\wt{F}$ then follows from the fact that $\wt{F}$ is finite modulo $\vp^{1-\epsilon}$, since it is the relative Frobenius of a morphism of finite presentation (see e.g. \cite[Tag 0CCD]{sta} for the case $q=p$). To prove that the generic fibre is finite flat of degree $q^{n-1}$, we first do the case $\epsilon = 0$. In this case $\Xb(0)$ is smooth of relative dimension $n-1$ (Remark \ref{ordinary locus}), so the relative Frobenius is finite and locally free of degree $q^{n-1}$ (see e.g. \cite[Proposition 3.2]{ill2} when $q=p$), and hence the same is true for $\wt{F}$ and its generic fibre. For general $\epsilon$, the generic fibre is a finite surjective morphism between smooth rigid spaces, hence flat. To compute the degree, we use that the diagram 
\begin{equation*}
\xymatrix{\Xc(0) \ar[r]\ar[d]^{\wt{F}} & \Xc(q^{-1}\epsilon) \ar[d]^{\wt{F}} \\
\Xc(0) \ar[r] & \Xc(\epsilon)}
\end{equation*}
is cartesian; then the right vertical morphism has the same degree as the left vertical morphism, which we already know has degree $q^{n-1}$.

\medskip
We now turn to part (2). The existence of canonical subgroups $C_{m}$ of level $m$ again follows from Proposition \ref{canonicalsubgroup}. The formula in the proposition then defines a morphism $\Xc(q^{-m}\epsilon) \ra \Xc_{U_{0}(\vp^{m})}$ by Proposition \ref{properties}(4). To see that it is an open immersion, we consider the map $\pi_{2}:\Xc_{U_{0}(\vp^{m})} \ra \Xc$ sending a pair $(A,C^{\prime})$ (with extra structures) to $A/D^{\prime}$ (with extra structures), where $D^{\prime}\sub A[p^{\infty}]$ corresponds to the $\ok$-subgroup
$$ (C^{\prime})^{\oplus n} \oplus A[w_{2}^{m}] \oplus \dots A[w_{r}^{m}] \oplus ((C^{\prime})^{\perp})^{\oplus n} \oplus 0 \oplus \dots \oplus 0. $$ 
The composition $\Xc(q^{-1}\epsilon) \ra \Xc_{U_{0}(\vp^{m})} \overset{\pi_{2}}{\ra} \Xc$ sends an abelian variety $A$ (with extra structures) to $A/A[p^{m}]$ (with extra structures) by direct computation. It follows that the composition is equal to the forgetful map $\Xc(q^{-1}\epsilon) \ra \Xc$ (which is an open immersion) followed by an isomorphism of $\Xc$ (which only changes the level structures away from $w$), and is hence an open immersion. Since $\pi_{2}$ is \'etale, it follows that $\Xc(q^{-1}\epsilon) \ra \Xc_{U_{0}(\vp^{m})}$ is an open immersion as desired.

\medskip
The commutation of the diagram in (2) follows from Proposition \ref{properties}. To see that it is cartesian we argue as follows. The horizontal maps are open embeddings, and the right vertical map is finite \'etale of degree $q^{n-1}$. Since the left vertical map is finite flat of degree $q^{n-1}$ by part (1), it follows that the induced map $\Xc(q^{-m-1}\epsilon) \ra \Xc(q^{-m}\epsilon) \times_{\Xc_{U_0(\vp^m)}}\Xc_{U_0(\vp^{m+1})}$ is a finite surjective morphism of degree $1$ between smooth rigid spaces, and hence an isomorphism. In particular, $\wt{F}$ is \'etale, and the diagram is cartesian. This finishes the proof of (2).

\medskip
For (3), we first need to establish that $\Xc(q^{-1}\epsilon) \ra \Xc_{U_{0}(\vp)}$ has image inside $\Xc_{U_{0}(\vp)}(\epsilon)_{a}$. This is done as in the last part of the proof of \cite[Theorem 3.2.15]{tor}. After this, we look at the diagram
\begin{equation*}
\xymatrix{\Xc(q^{-1}\epsilon) \ar[r]\ar[d]^{\wt{F}} & \Xc_{U_{0}(\vp)}(\epsilon)_{a} \ar[d] \\
\Xc(\epsilon) \ar[r]^{id} & \Xc(\epsilon).}
\end{equation*}
As in the proof of part (2), it commutes. We claim that it is cartesian; since the bottom horizontal arrow is the identity this gives the desired conclusion. The left vertical map is finite of degree $q^{n-1}$, and one checks that the right vertical map is finite \'etale of degree $q^{n-1}$. An argument as in the proof of (2) then shows that the diagram is cartesian, and finishes the proof. 
\end{proof}

For the next result, which is the main result of this subsection, we use the notion $X\sim \varprojlim_{i}X_{i}$ for an adic space $X$ with a collection of compatible maps to a cofiltered inverse system of adic spaces $(X_i)$ from \cite[Definition 2.4.1]{sw}.

For $m\geq 1$ we define $\Xc_{U_{0}(\vp^{m})}(\epsilon)_{a}$ as the image of $\Xc(q^{-m}\epsilon)$ in $\Xc_{U_{0}(\vp^{m})}$.
\begin{theorem}\label{epsilonnhood}
Fix $0\leq \epsilon < 1/2$. There is a unique (affinoid) perfectoid space $\Xc_{P(\ok)}(\epsilon)_{a}$ over $K^{cycl}$ such that
$$ \Xc_{P(\ok)}(\epsilon)_{a} \sim \varprojlim_{m}\Xc_{U_{0}(\vp^{m})}(\epsilon)_{a}. $$
\end{theorem}

\begin{proof}
We start by showing the existence of such a perfectoid space $\Xc_{P(\ok)}(\epsilon)_{a}$. By Theorem \ref{anticanonical} we may identify the tower $(\Xc_{U_{0}(\vp^{m})}(\epsilon)_{a})_{m \geq 0}$ with $(\Xc(q^{-m}\epsilon))_{m\geq 0}$, with transition maps given by $\wt{F}$. This gives us a formal model $(\wh{\Xf}(q^{-m}\epsilon))_{m \geq 0}$ for this tower, and we may take the inverse limit
$$ \wh{\Xf}_{\infty} := \varprojlim_{m\geq 0} \wh{\Xf}(q^{-m}\epsilon) $$
in the category of $\vp$-adic formal schemes since the transition maps are affine. We define $\Xc_{P(\ok)}(\epsilon)_{a}$ to be the generic fibre of $\wh{\Xf}_{\infty}$ in the sense of \cite[\S 2.2]{sw}. Since the transition maps agree with Frobenius modulo $\vp^{1-\epsilon}$, we may argue as in the proof of \cite[Corollary 3.2.19]{tor} to conclude that $\Xc_{P(\ok)}(\epsilon)_{a}$ is perfectoid and that $ \Xc_{P(\ok)}(\epsilon)_{a} \sim \varprojlim_{m}\Xc_{U_{0}(\vp^{m})}(\epsilon)_{a}$.

\medskip
Finally, to show that $\Xc_{P(\ok)}(\epsilon)_{a}$ is affinoid perfectoid, one may argue using tilts as in \cite[Corollary 3.2.19, Corollary 3.2.20]{tor}. Since this additional information is not needed for the results of this paper we will not give further details.
\end{proof}

\subsection{The Hodge--Tate period map}

We now introduce some notation for more general `infinite level Shimura varieties'. These will be defined (a priori) as diamonds, and we refer to \cite{dia} for the definitions and terminology concerning diamonds. Let $H_v \sub \GL_n(\oo_K)$ be a closed subgroup. We define
\[
 \Xc_{H_{v}} := \varprojlim_{H_{v}\sub U_{v}} \Xc_{U_{v}}^{\lozenge},
\]
where $U_{v}$ ranges through all the open subgroups $U_{v}\sub \GL_n(\oo_K)$ containing $H_v$, and $Y \mapsto Y^{\lozenge}$ is the `diamondification functor' on rigid spaces \cite[Definition 15.5]{dia}. We remark that each $\Xc_{U_v}^{\lozenge}$ is a spatial diamond, and that the inverse limits above exist (as diamonds) and are spatial by \cite[Lemma 11.22]{dia}, which also says that the natural map
\[
 |\Xc_{H_{v}}| \ra \varprojlim_{H_{v}\sub U_{v}} |\Xc_{U_{v}}^{\lozenge}| = \varprojlim_{H_{v}\sub U_{v}} |\Xc_{U_{v}}|
\]
is a homeomorphism, where $|Y|$ denotes the underlying topological space of an adic space or a diamond \cite[Definition 11.14]{dia} (and the equality  follows from \cite[Lemma 15.6]{dia}). Note that if $H_v=U_v$ is open, our definition above is essentially saying that we will conflate $\Xc_{U_v}$ with its corresponding diamond; this abuse of notation is mostly harmless since the diamondification functor is fully faithful on the category of normal rigid spaces (over a fixed nonarchimedean field, remembering the structure morphism). 

\medskip

Thus, writing $\mbf{1} \sub \GL_n(\oo_K)$ for the trivial subgroup, we have a diamond $\Xc_{\mbf{1}} = \varprojlim \Xc_{U_{v}}$ with an action of $\GL_{n}(\oo_{K})$, which extends to an action of $\GL_{n}(K)$ by using the maps $g : \Xc_{gU_{v}g^{-1}} \ra \Xc_{U_v}$ for $g\in \GL_n(K)$ and any open $U_v$ such that $U_v , gU_v g^{-1} \sub \GL_n(\oo_K)$. Our goal is to show that a certain open subset $\Xc_{P(\ok)}^{comp} \sub \Xc_{P(\ok)}$ (containing $\Xc_{P(\ok)}(\epsilon)_{a}$ for sufficiently small $\epsilon>0$) is perfectoid. 

\medskip
To do this, we proceed from the previous subsection by going further up the tower. Recall that if $(Y_i)_{i\in I}$ is a filtered inverse system of adic spaces over a perfectoid field with qcqs transition maps and $Y$ is a perfectoid space with compatible maps $Y \to Y_i$ such that $Y \sim \varprojlim_i Y_i$, then by \cite[Proposition 2.4.5]{sw} and the definition of the diamondification functor we have $Y = \varprojlim_i Y_i^\lozenge$ as diamonds (here and elsewhere, if $Y$ is a perfectoid space, we simply write $Y$ for the corresponding diamond as well). Thus, by Theorem \ref{epsilonnhood}, we have
\[
 \Xc_{P(\ok)}(\epsilon)_{a} = \varprojlim_{m\geq 0} \Xc_{U_{0}(\vp^{m})}(\epsilon)_{a}^{\lozenge},
\]
and $\Xc_{P(\ok)}(\epsilon)_{a}$ is naturally an open subdiamond of $\Xc_{P(\ok)}$.

\begin{proposition}\label{epsilonnhood at full inf level}
Let $0\leq \epsilon <1/2$ and let $H_{v}\sub \GL_{n}(\ok)$ be a closed subgroup contained in $P(\ok)$. Then the spatial diamond $\Xc_{H_{v}}(\epsilon)_{a} := \Xc_{P(\ok)}(\epsilon)_{a} \times_{\Xc_{P(\ok)}}\Xc_{H_{v}}$ is an (affinoid) perfectoid space.
\end{proposition}

\begin{proof}
First assume that $H_{v}$ has finite index inside $P(\ok)$. Then $\Xc_{H_{v}}(\epsilon)_{a} \ra \Xc_{P(\ok)}(\epsilon)_{a}$ is finite \'etale, and the result then follows. In general $\Xc_{H_{v}}(\epsilon)_{a} = \varprojlim_{H^{\prime}_{v}} \Xc_{H_{v}^{\prime}}(\epsilon)_{a}$ where $H_{v}^{\prime}$ ranges over closed subgroups with $H_{v} \sub H_{v}^{\prime} \sub P(\ok)$ and $H_{v}^{\prime} \sub P(\ok)$ has finite index, and the result follows.
\end{proof}

\medskip
To continue, we construct the Hodge--Tate period map $\Xc_{\mbf{1}} \ra (\Pro)^{\lozenge}$ on diamonds; this is the content of the following proposition. We keep the statement vague; the meaning of the name `Hodge--Tate period map' should be clear from the construction.

\begin{proposition}
There exists a $\GL_{n}(K)$-equivariant Hodge--Tate period map $\pi_{HT} : \Xc_{\mbf{1}} \ra (\Pro)^{\lozenge}$ over $(K^{cycl},\ok^{cycl})$.
\end{proposition}

\begin{proof}
By the definitions, we may regard $\Xc_{\mbf{1}}$ and $(\Pro)^{\lozenge}$ as sheaves on the pro-\'etale site of perfectoid spaces over $K^{cycl}$, so to construct a map of sheaves it suffices to work with a basis for the topology. Let $\Spa(R,R^{+})$ be a strictly totally disconnected perfectoid space over $(K^{cycl},\ok^{cycl}$). A map $\Spa(R,R^{+}) \ra \Xc_{\mbf{1}}$ is the same as a compatible system of maps $\Spa(R,R^{+}) \ra \Xc_{U(\vp^m)}$ for all $m$, and we may assume that the map $\Spa(R,R^{+}) \ra \Xc$ factors through an affinoid open subset $\Spa(A,A^{\circ}) \sub \Xc$, where $\Spf(A^{\circ}) \sub \wh{\Xf}$ is open affine (note that this is possible since $\wh{\Xf}$ is normal, by \cite[Theorem 7.4.1]{dj}). The map $\Spa(R,R^{+}) \ra \Spa(A,A^{\circ})$ is then the generic fibre of a map $\Spf(R^{+}) \to \Spf(A^{\circ})$ of $\vp$-adic formal schemes, and we may pull back the universal Barsotti--Tate $\ok$-module over $\Spf(A^{\circ})$ to a Barsotti--Tate $\ok$-module $\G_{R}$ over $R^{+}$. Since $\Spa(R,R^+)$ is strictly totally disconnected, we may apply \cite[Proposition 4.3.6]{sw}\footnote{The proof of \cite[Proposition 4.3.6]{sw} does not require the assumption, in the notation of that reference, that $\Spec T$ is connected.} to see that $\G_{R}$ has an exact Hodge--Tate sequence
$$ 0 \ra \Lie(\G_R)(1)\otimes_{R^{+}}R \ra T\G_R(R^{+})\otimes_{\Zp}R \ra (\Lie(\G_R^\vee))^{\vee}\otimes_{R^{+}}R \ra 0 $$
of finite projective $R$-modules. By the compatibility of $\G_R$ and the fact that it has dimension $1$, 
\[\Lie(\G_R)(1)\otimes_{R^{+}}R\]
has $R$-rank $1$ and embeds into $T\G_R(R^{+})\otimes_{\ok^{cycl}}R$ (which is an $R$-module direct summand of $T\G_R(R^{+})\otimes_{\Zp}R$). Using the compatible trivialisations $\G_{R}[\vp^{m}](R^{+})=\G_R[\vp^m](R) \cong (\ok/\vp^{m})^{n}$ coming from the maps $\Spa(R,R^{+}) \ra \Xc_{U(\vp^m)}$, the inclusion $\Lie(\G_R)(1)\otimes_{R^{+}}R \sub T\G_R(R^{+})\otimes_{\ok^{cycl}}R \cong R^{n}$ defines an $(R,R^{+})$-point of $\Pro$. This gives the desired map, and $\GL_n(K)$-equivariance is clear from the construction.
\end{proof}

We remark that any map between spatial diamonds induces a spectral map of the underlying spectral topological spaces, so $\pi_{HT}$ is spectral. The next lemma characterises the image of the $\mu$-ordinary locus under the Hodge--Tate period map. For more general results under the assumption that $K/\Qp$ is unramified, see \cite[\S 11]{he}.

\begin{lemma}\label{mu-ordinary vs rational}
Let $C$ be a complete algebraically closed extension of $K$ with valuation ring $\oo_C$ and residue field $k_C$. Let $\G$ be a compatible Barsotti--Tate $\ok$-module over $\oo_C$ of dimension $1$ and height $n$. Then the special fibre $\G \times_{\oo_C}k_C$ is $\mu$-ordinary if and only if the subspace $\Lie(\G)\otimes_{\oo_C}C(1) \sub T\G \otimes_{\Zp}C$ is $K$-rational (here $T\G$ is the Tate module of $\G$).
\end{lemma}

\begin{proof}
We use the Scholze--Weinstein classification of Barsotti--Tate groups over $\oo_C$ \cite[\S 5]{sw}. To simplify the notation, we will take the linear algebra data $(T,W)$ in the Scholze--Weinstein equivalence \cite[Theorem 5.2.1]{sw} to be a finite free $\Zp$-module $T$ together with a $C$-subspace $W \sub T\otimes_{\Zp}C$ rather than a subspace of $T\otimes_{\Zp}C(-1)$ (from the point of view of Barsotti--Tate groups $G$, we take $W$ to be $\Lie(G)\otimes_{\oo_C}C(1)$ rather than $\Lie(G)\otimes_{\oo_C}C$). We start by assuming that the special fibre of $\G$ is $\mu$-ordinary, and consider the connected-\'etale sequence  
$$ 0 \ra \G^{0} \ra \G \ra \G^{et} \ra 0 $$
of $\G$, which is an exact sequence of compatible Barsotti--Tate $\ok$-modules. By \cite[Proposition 5.2.8]{sw}, this exact sequence induces an exact sequence
$$ 0 \ra T\G^{0} \ra T\G \ra T\G^{et} \ra 0 $$
and an equality $\Lie(\G^{0})\otimes_{\oo_C}C = \Lie(\G)\otimes_{\oo_C}C$ (since $\Lie(\G^{et})=0$), so it suffices to show that $\Lie(\G^{0})\otimes_{\oo_C}C(1) \sub T\G^{0}$ is $K$-rational. Since the special fibre is $\mu$-ordinary, $\G^{0}$ has height $1$ (using that the connected-\'etale sequence is compatible with reduction). But, by the Scholze--Weinstein classification \cite[Theorem 5.2.1]{sw}, there is a unique compatible Barsotti--Tate $\ok$-module $LT$ of dimension $1$ and height $n$ over $\oo_C$, given by the linear algebra datum $(T=\ok, W=C_{\sigma})$ where
$$T\otimes_{\Zp}C = \prod_{\tau \in \Hom(K,C)}C_{\tau} $$
and $\sigma : K \ra C$ is the inclusion (recall that $C$ was defined to be an extension of $K$), and this $W$ is visibly $K$-rational. Note that $LT$ is the unique lift of the Lubin--Tate $\ok$-module of height $1$ over $k_C$.

\medskip
For the converse, assume that $(T=\ok^{n},W)$ is the linear algebra datum of $\G$, assume that $W$ is $K$-rational and use the notation established in the previous paragraph. Write $W=W_{K}\otimes_{K}C$ with $W_{K}$ a $K$-rational structure on $W$. We can canonically identify $T[1/p]=K^{n}$ with the $K$-rational structure on $(T\otimes_{\Zp}C)_{\sigma}$ and hence think of $W_{K}$ as a subspace of $T[1/p]$; the intersection $W_{\oo_{K}}=W_{K} \cap T$ is then an $\ok$-module direct summand of $T$  of rank $1$; let $T^{\prime}\sub T$ be a complement. It follows that we can write
$$ (T,W) = (W_{\oo_{K}},W) \oplus (T^{\prime},0) $$
compatibly with the $\ok$-action. It then follows from the Scholze--Weinstein equivalence that $\G$ is isomorphic to $LT \times (K/\ok)^{n-1}$ as a Barsotti--Tate $\ok$-module, and hence has $\mu$-ordinary reduction.
\end{proof}

Let us now define 
$$ \Pro(K)_{a} := \left\{ (a_{1}: \dots :a_{n}) \in \Pro(\ok) \mid a_{n}\in \ok^{\times} \right\}.$$
We then get the following corollaries.

\begin{corollary}\label{image of antican locus}
We have $\pi_{HT}(|\Xc_{\mbf{1}}(0)_{a}|) = \Pro(K)_{a}$ and $\pi_{HT}^{-1}(\Pro(K)_{a})$ is equal to the closure $\ol{|\Xc_{\mbf{1}}(0)_{a}|}$ of $|\Xc_{\mbf{1}}(0)_{a}|$ in $|\Xc_{\mbf{1}}|$.

\end{corollary}

\begin{proof}
Form now on, to ease the notation we will often drop the $|-|$ when discussing topological spaces of adic spaces or diamonds; what we mean should hopefully be clear from the context. By Lemma \ref{mu-ordinary vs rational} the rank one points of $\pi_{HT}^{-1}(\Pro(K))$ are precisely the rank one points of the $\mu$-ordinary locus $\Xc_{\mbf{1}}(0)$, so it follows that $\pi_{HT}^{-1}(\Pro(K))$ is precisely the set of specializations of points in $\Xc_{\mbf{1}}(0)$. Since $\Xc_{\mbf{1}}(0)$ is a quasicompact open subset of $\Xc_{\mbf{1}}$, the set of such specializations is precisely $\ol{\Xc_{\mbf{1}}(0)}$. Moreover $\Xc_{\mbf{1}}(0)$ is $\GL_n(\ok)$-stable and $\Pro(K)$ is a $\GL_n(\ok)$-orbit, so by equivariance of $\pi_{HT}$ the image of $\Xc_{\mbf{1}}(0)$ has to be all of $\Pro(K)$. Finally, to deduce the corollary from this one checks easily that the anticanonical condition on a rank $1$ point is equivalent to the image under $\pi_{HT}$ being in $\Pro(K)_{a}$, and then we argue similarly using that $\Xc_{\mbf{1}}(0)_{a}$ is also quasicompact and open.
\end{proof}

\begin{corollary}\label{nhood}
For every $0< \epsilon <1/2$ there exists a quasicompact open subset $U \sub \Pro$ containing $\Pro(K)_{a}$ such that $\pi_{HT}^{-1}(U) \sub \Xc_{\mbf{1}}(\epsilon)_{a}$. Conversely, for every open subset $V\sub \Pro$ containing $\Pro(K)_{a}$, we have $\Xc_{\mbf{1}}(\epsilon)_{a} \sub \pi_{HT}^{-1}(V)$ for all sufficiently small $\epsilon >0$.
\end{corollary}

\begin{proof}
We may write $\Pro(K)_{a} = \bigcap U$, where $U$ runs through the quasicompact open subsets of $\Pro$ containing $\Pro(K)_{a}$. Fix $\epsilon >0$ small enough. We have $\ol{\Xc_{\mbf{1}}(0)_{a}} \sub \Xc_{\mbf{1}}(\epsilon)_{a}$, so by Corollary \ref{image of antican locus} we have $  \Xc_{\mbf{1}}(\epsilon)_{a} \supseteq \bigcap \pi^{-1}_{HT}(U) $, and it follows (by a short argument using the constructible topology) that $\pi_{HT}^{-1}(U) \sub \Xc_{\mbf{1}}(\epsilon)_{a}$ for some $U$ since the $\pi^{-1}_{HT}(U)$ are quasicompact opens (since $\pi_{HT}$ is spectral). This proves the first part, and the converse is proved in exactly the same way using the fact that $\ol{\Xc_{\mbf{1}}(0)_{a}} = \bigcap_{\epsilon >0} \Xc_{\mbf{1}}(\epsilon)_{a}$.
\end{proof}

\subsection{Perfectoid spaces}\label{perfectoid}

In this subsection we will prove the (global) perfectoidness results that we will need in this paper. We start with some remarks on the geometry of $\Pro$, to set up notation. We have a cover of $\Pro$ by open affinoid subsets
$$ \V_{i} = \{ (a_{1}: \dots :a_{n}) \mid |a_{j}|\leq |a_{i}|\,\,j\neq i \}. $$
Note also that the $\V_{i}$ are translates of one another under the action of the Weyl group of $\GL_n$ (with respect to the diagonal torus). We have a similar `algebraic' cover by open subsets
$$ V_{i} = \{ (a_1 : \dots :a_n) \mid |a_i|\neq 0 \}. $$
Let $ \gamma = {\rm diag}(\vp,\dots,\vp,1) \in \GL_n(K)$. We then have the following elementary lemma. Recall that we are using the right action of $\GL_n$ on $\Pro$ which is the inverse of the usual left action.

\begin{lemma}\label{cover of pro}
We have $V_n = \bigcup_{k \geq 0} \V_{n} \gamma^{k}$, and the sets $\V_{n}\gamma^{-k}$, $k \geq 0$, form a basis of quasicompact open neighbourhoods of $(0:\dots 0:1) \in \Pro$.
\end{lemma}

Next, we define $\Xc_{\mbf{1}}^{comp}$, the `complementary locus', to be the open subdiamond $\pi_{HT}^{-1}(V_{n}) \sub \Xc_{\mbf{1}}$.

\begin{corollary}\label{comp locus at infinite level}
Let $\epsilon >0$ be sufficiently small. We have $\Xc_{\mbf{1}}^{comp}= \bigcup_{k \geq 0}\Xc_{\mbf{1}}(\epsilon)_{a}\gamma^{k}$, and hence $\Xc_{\mbf{1}}^{comp}$ is a perfectoid space.
\end{corollary}

\begin{proof}
By Corollary \ref{nhood} and the second part of Lemma \ref{cover of pro} we can choose a $U$, and $\epsilon >0$ and a $k\geq 0$ such that $\pi_{HT}^{-1}(U) \sub \Xc_{\mbf{1}}(\epsilon)_{a} \sub \Xc_{\mbf{1}}^{comp}$ and $\V_{n}\gamma^{-k} \sub U$. The first assertion of this corollary then follows from the first part of Lemma \ref{cover of pro} (using the equivariance of $\pi_{HT}$), and the second part of the corollary is immediate from the first and Proposition \ref{epsilonnhood at full inf level}.
\end{proof}

As an aside, which won't be used in this paper, we note the following theorem.

\begin{theorem}\label{full}
$\Xc_{\mbf{1}}$ is a perfectoid space and $\pi_{HT}$ comes from a unique map $\Xc_{\mbf{1}} \ra \Pro $ of adic spaces.
\end{theorem}

\begin{proof}
The fact that $\Xc_{\mbf{1}}$ is a perfectoid space follows Corollary \ref{comp locus at infinite level} and the fact that
$$ |\Xc_{\mbf{1}}| = \bigcup_{g\in \GL_n(\ok)} |\Xc^{comp}_{\mbf{1}}|g $$
(which is immediate from equivariance of $\pi_{HT}$ and $\Pro = V_{1} \cup \dots \cup V_n$). The second part then follows immediately, since any map of diamonds from a perfectoid space $S$ to the diamond $Z^{\lozenge}$ of a rigid space $Z$ corresponds to a unique map of adic spaces $S \ra Z$, by the definition of the diamondification functor. 
\end{proof}

We now turn to the main result of this section. The natural map $ |\Xc_{\mbf{1}}| \ra |\Xc_{P(\ok)}|$ is open, so we may define $\Xc_{P(\ok)}^{comp} \sub \Xc_{P(\ok)}$ to be the open subdiamond given as the image of $|\Xc_{\mbf{1}}^{comp}|$. Note that $\Xc_{\mbf{1}}^{comp}$ is $P(\ok)$-stable. 
From the next lemma, we see that $\Xc_{\mbf{1}}^{comp} \ra \Xc_{P(\ok)}^{comp}$ is a $P(\ok)$-torsor. 

\begin{lemma}\label{torsor}
Assume that $H_{v}^{\prime} \sub H_{v}$ are closed subgroups of $\GL_n(\oo_K)$, and that $H_{v}^{\prime}$ is normal in $H_v$. Then $\Xc_{H_{v}^{\prime}} \ra \Xc_{H_v}$ is a $H_{v}/H_{v}^{\prime}$-torsor in the sense of \cite[Definition 10.12]{dia}.
\end{lemma}
\begin{proof}
Set $U_{v,m}=H_{v}U(\vp^m)$, $U_{v,m}^{\prime}=H_{v}^{\prime}U(\vp^m)$. Then $\Xc_{U_{v,m}^{\prime}} \ra \Xc_{U_{v,m}}$ is a $U_{v,m}/U_{v,m}^{\prime}$-torsor, compatibly in $m$. Diamondification preserves torsors by finite groups, so we have compatible isomorphisms
$$ \Xc_{U_{v,m}^{\prime}} \times \underline{U_{v,m}/U_{v,m}^{\prime}} \overset{\sim}{\longrightarrow} \Xc_{U_{v,m}^{\prime}} \times_{\Xc_{U_{v,m}}} \Xc_{U_{v,m}^{\prime}} $$
for all $m$. Taking the inverse limit over $m$ then gives the result.
\end{proof}

Let us recall that Huber defined the category of adic spaces as a full subcategory of a category he called $\V$ in \cite{hu}. This category has quotients by arbitrary group actions, cf. \cite[\S 2.2]{lud}. Let us explicitly record the following link between torsors and group quotients in $\V$, in the case of perfectoid spaces.

\begin{lemma}\label{quotients and torsors}
Let $H$ be a profinite group and let $\wt{X} \ra X$ be a map of perfectoid spaces which is a $H$-torsor in the sense of \cite[Definition 10.12]{dia}. Then $X$ is the quotient of $\wt{X}$ by $H$ in the category $\V$.
\end{lemma}

\begin{proof}
It suffices to check that $|\wt{U}|/H =|U|$ and $\oo_{\wt{X}}(\wt{U})^{H}=\oo_{X}(U)$ for a basis of open subsets $U$ of $X$, with $\wt{U}:= \wt{X} \times_{X}U$. So take $U=\Spa(R,R^{+})\sub X$ affinoid perfectoid. By \cite[Lemma 10.13]{dia} we may then write $\wt{U}$ as an inverse limit $\wt{U}= \varprojlim_{K}\wt{U}_{K} \ra U$ of finite \'etale (and hence affinoid perfectoid) $\wt{U}_{K} \ra U$ for open normal subgroups $K\sub H$ which are $H/K$-torsors. Write $\wt{U}_{K}=\Spa(R_{K},R_{K}^{+})$ and $\wt{U} =\Spa(S,S^{+})$. If $\pi$ is a pseudouniformizer for $R$ we then have $(R_{K}^{+}/\pi^{m})^{H/K} =^{a} R^{+}/\pi^{m}$ for all $m$, compatibly in $K$ ($=^{a}$ for almost equal). This implies that $S^{H}=R$ and that $(R_{K}^{H/K},(R_{K}^{+})^{H/K})=(R,R^{+})$ compatibly in $K$. The latter implies that $|\wt{U}_{K}|/(H/K)=|U|$ compatibly in $K$ (e.g. by \cite[Theorem 1.2]{ha}) which implies that $|\wt{U}|/H=|U|$ as desired. 
\end{proof}

\begin{theorem}\label{comp locus at gamma 0 level}
$\Xc_{P(\ok)}^{comp}$ is a perfectoid space. More precisely, for $\epsilon >0$ sufficiently small, $|\Xc^{comp}_{P(\ok)}|$ is covered by the open subsets $|\Xc_{\mbf{1}}(\epsilon)_{a} \gamma^{k}|/P(\ok)$ for $k\geq 0$, and the corresponding open subdiamonds are (affinoid) perfectoid spaces. Moreover, $\Xc_{P(\ok)}^{comp}$ is the quotient of $\Xc_{\mbf{1}}^{comp}$ by $P(\ok)$ in Huber's category $\V$.
\end{theorem}

\begin{proof}
We have an isomorphism 
$$ \gamma^{-k} : \Xc_{P(\ok)} \ra \Xc_{\gamma^{k}P(\ok)\gamma^{-k}} $$
of diamonds, which sends the open subset $|\Xc_{\mbf{1}}(\epsilon)_{a} \gamma^{k}|/P(\ok)$ of $|\Xc_{P(\ok)}|$ to the open subset 
$$|\Xc_{\mbf{1}}(\epsilon)_{a}|/\gamma^{k}P(\ok)\gamma^{-k}$$
 of $|\Xc_{\gamma^{k}P(\ok)\gamma^{-k}}|$. Let us denote the open subdiamond corresponding to $|\Xc_{\mbf{1}}(\epsilon)_{a}|/\gamma^{k}P(\ok)\gamma^{-k}$ by $\Xc_{\gamma^{k}P(\ok)\gamma^{-k}}(\epsilon)_{a}$. By direct computation $\gamma^{k}P(\ok)\gamma^{-k}$ is a finite index open subgroup of $P(\ok)$, so we have a natural finite \'etale map $\Xc_{\gamma^{k}P(\ok)\gamma^{-k}}(\epsilon)_{a} \ra \Xc_{P(\ok)}(\epsilon)_{a}$. It follows that $\Xc_{\gamma^{k}P(\ok)\gamma^{-k}}(\epsilon)_{a}$ is (affinoid) perfectoid, and hence that the diamond corresponding to $|\Xc_{\mbf{1}}(\epsilon)_{a} \gamma^{k}|/P(\ok)$ is (affinoid) perfectoid. This proves the theorem, except for the `moreover' part, which then follows from Lemma \ref{quotients and torsors} since $\Xc^{comp}_{\mbf{1}} \ra \Xc_{P(\ok)}^{comp}$ is a $P(\ok)$-torsor.
\end{proof}

\section{The Lubin--Tate tower}\label{sec:LT}

In this section we prove our geometric results on the Lubin--Tate tower.

\subsection{Preliminaries}

We begin by recalling the Lubin--Tate spaces that we will be working with, cf. \cite{gh,rz}. Let $\G_0$ be the unique one-dimensional compatible Barsotti--Tate $\ok$-module of $\ok$-height $n$ and with $\G_{0}^{et}=0$ over $\ol{k}$, and set $\breve{K}=K\otimes_{W(k)}W(\ol{k})$. The Lubin--Tate space $\mf{M}$ is the formal scheme over $\oo_{\breve{K}}$ whose $R$-points, for $R$ an $\oo_{\breve{K}}$-algebra with $\vp$ nilpotent, is the set of pairs $(\G,\rho)$ where $\G$ is a one-dimensional compatible Barsotti--Tate $\ok$-module over $R$ and $\rho : \G_0 \otimes_{\ol{k}} R/\vp \ra \G \otimes_{R} R/\vp $ is an $\ok$-linear quasi-isogeny. $\mf{M}$ decomposes as a disjoint union
$$ \mf{M} = \bigsqcup_{d\in \Z} \mf{M}^{(d)} $$
according to the degree $q^d$ of the quasi-isogeny $\rho$, and $\rho$ is an isomorphism if $d=0$. In particular, $\mf{M}^{(0)}$ is the formal deformation space of $\G_0$. Let $\mc{M}$ and $\mc{M}^{(d)}$ be the generic fibre of $\mf{M}$ and $\mf{M}^{(d)}$, respectively. There is a tower of rigid analytic varieties $(\mc{M}_U)_{U}$ over $\mc{M}=\mc{M}_{\GL_n(\ok)}$, where $U$ ranges over the open subgroups of $\GL_n(\ok)$. All transition maps are finite \'etale, and the tower carries an action of $\GL_n(K)$. We also set $\mc{M}^{(d)}_U := \mc{M}_{U} \times_{\mc{M}}\mc{M}^{(d)}$ for all $d\in \Z$. Similarly to our notation for Shimura varieties in the previous section, we set
$$ \mc{M}_{H} := \varprojlim_{U \supseteq H} \mc{M}_{U}^{\lozenge}$$
for any closed subgroup $H\sub \GL_n(\ok)$; here $U$ ranges over the open subgroups containing $H$ (we define $\mc{M}_{H}^{(d)}$ similarly). We have two period maps; the Gross--Hopkins period map $\pi_{GH} : \mc{M}_{\GL_n(\ok)} \ra \Pro$ and the Hodge--Tate period map $\pi_{HT} : \mc{M}_{\mbf{1}} \ra \Pro $. The map $\pi_{GH}$ is \'etale, surjective and admits local sections\footnote{When $K=\Qp$ this is a special case of \cite[Lemma 6.1.4]{sw}, but the argument there works in general.}. Moreover, the composite
$$ \mc{M}_{\mbf{1}} \ra \mc{M}_{\GL_n(\ok)} \overset{\pi_{GH}}{\longrightarrow} \Pro $$
is a $\GL_n(K)$-torsor in the sense of \cite[Definition 10.12]{dia}. The image of the Hodge--Tate period map $\pi_{HT}$ is the Drinfeld upper halfspace $\Omega^{n-1}\sub \Pro$.

\medskip
We now relate our Lubin--Tate spaces to the Shimura varieties from the previous section. We use the notation and conventions of the previous sections freely, except that we will base change all analytic adic spaces to a complete and algebraically closed non-archimedean field extension $C$ of $K$ (e.g. $\C_p$), all formal schemes to $\oo_C$, and all reductions to the residue field $k_C$ of $C$ or $\oo_C/\vp$ as appropriate. Then, we choose once and for all a closed point $x$ in $\Xb^{(0)}$ (which is non-empty by \cite[Lemma III.4.3]{ht}). By \cite[Lemma III.4.1(1)]{ht}, this realises $\mf{M}^{(0)}$ as the completed local ring of $\Xf$ at $x$. Taking generic fibres, we obtain an open immersion
$$ \mc{M}^{(0)} \hookrightarrow \Xc $$
and taking level structures we obtain compatible embeddings
$$ \mc{M}_{U}^{(0)} \hookrightarrow \Xc_U $$
for all open subgroups $U \sub \GL_n(\ok)$, and this map of towers is compatible with the Hecke actions. Taking inverse limits (as diamonds), we get more generally open immersions
$$ \mc{M}_{H}^{(0)} \hookrightarrow \Xc_H $$
for all closed subgroups $H \sub \GL_n(\ok)$. The fact that these are open immersions follows from the fact that $\mc{M}_{U}^{(0)} = \mc{M}^{(0)} \times_{\Xc} \Xc_U$ for all open $U \sub \GL_n(\ok)$ (and this identity then extends to all closed $H \sub \GL_n(\ok)$). We also have a compatibility between the local and global Hodge--Tate period maps: Composing the immersion $\mc{M}_{\mbf{1}} \hookrightarrow \Xc_{\mbf{1}}$ and the global $\pi_{HT} : \Xc_{\mbf{1}} \ra \Pro $ gives the local $\pi_{HT} : \mc{M}_{\mbf{1}} \ra \Pro $. Since the Drinfeld upper halfspace $\Omega^{n-1}$ is contained in the `complementary locus' $V_n \sub \Pro$ from Subsection \ref{perfectoid}, we obtain $\mc{M}_{\mbf{1}}^{(0)} \sub \Xc_{\mbf{1}}^{comp}$ and hence $\mc{M}_{P(\ok)}^{(0)} \sub \Xc_{P(\ok)}^{comp}$. Theorems \ref{comp locus at infinite level} and \ref{comp locus at gamma 0 level}, together with Lemma \ref{torsor} then directly imply the following local analogue.

\begin{proposition}\label{step1}
$\mc{M}_{\mbf{1}}^{(0)}$ and $\mc{M}_{P(\ok)}^{(0)}$ are perfectoid spaces over $C$, and $\mc{M}_{P(\ok)}^{(0)}$ is the quotient of $\mc{M}_{\mbf{1}}^{(0)}$ by $P(\ok)$ in Huber's category $\V$.
\end{proposition}

\subsection{The main result}

We now turn to the task of showing that $\mc{M}_{P(K)} := \mc{M}_{\mbf{1}}/P(K)$ is a quasicompact perfectoid space, which is the main result of this section. This will follow from Proposition \ref{step1} precisely as in \cite[\S 3.6]{lud} in the case $n=2$, $F=\Qp$. To clarify, the quotient above is taken in the category $\V$; this makes sense since $\mc{M}_{\mbf{1}}$ is a perfectoid space (using Proposition \ref{step1} and the $\GL_n(K)$-action). Set 
$$ G^{\prime} := \{ g \in \GL_n(K) \mid \det(g) \in \ok^{\times} \}. $$
This is the kernel of the homomorphism $\GL_n(K) \ra \Z$ given by $g \mapsto v_K(\det(g))$, where $v_K$ is the normalised valuation on $K$, and this homomorphism is split. Moreover, for $g\in \GL_n(K)$, one has
$$ \mc{M}_{\mbf{1}}^{(0)}.g = \mc{M}_{\mbf{1}}^{(v_K(\det(g)))} $$
by looking at the degree of the quasi-isogeny. From this we see that $G^{\prime}$ is the stabiliser of the component $\mc{M}_{\mbf{1}}^{(0)}$, and it also follows that the natural map
$$ \mc{M}_{\mbf{1}}^{(0)}/P^\prime \ra \mc{M}_{\mbf{1}}/P(K)=\mc{M}_{P(K)} $$
in $\V$, where $P^\prime := P(K) \cap G^\prime$, is an isomorphism.

\begin{theorem}\label{main}
The quotient $\mc{M}_{P(K)}$ is a perfectoid space over $C$. The natural map $\mc{M}_{P(\ok)}^{(0)}\rightarrow \mc{M}_{P(K)}$ has local sections.
\end{theorem}

\begin{proof}
We follow the proof of \cite[Theorem 3.14]{lud}, indicating the details. By the remarks above $ \mc{M}_{P(K)} \cong \mc{M}_{\mbf{1}}^{(0)}/P^\prime$, so it suffices to show that the latter is a perfectoid space. Let $pr : \mc{M}_{\mbf{1}}^{(0)} \ra \mc{M}^{(0)}$ denote the map that forgets level structures, and let $U \sub \mc{M}^{(0)}$ be an open subset such that the Gross--Hopkins period map $\pi_{GH}|_U$ restricted to $U$ is an isomorphism onto its image $U$. The preimage $pr^{-1}(U)\sub \mc{M}_{\mbf{1}}^{(0)}$ is stable under $P(\ok)$, so we may form the object
$$ pr^{-1}(U) \times^{P(\ok)} P^\prime := (pr^{-1}(U) \times \ul{P}^\prime )/P(\ok) \in \V; $$
we refer to \cite[\S 2.4]{lud} for the details of this construction. By \cite[Lemma 2.16]{lud} and the way we have chosen $U$, there is an open immersion
$$ pr^{-1}(U) \times^{P(\ok)} P^\prime \hookrightarrow \mc{M}_{\mbf{1}}^{(0)}. $$
Since taking quotients is compatible with open immersions by construction, we get an open immersion
$$ (pr^{-1}(U) \times^{P(\ok)} P^\prime)/P^\prime \hookrightarrow \mc{M}_{\mbf{1}}^{(0)}/P^\prime. $$
By \cite[Proposition 2.14]{lud}, $(pr^{-1}(U) \times^{P(\ok)} P^\prime)/P^\prime \cong pr^{-1}(U)/P(\ok)$ and the latter is an open subset of $\mc{M}^{(0)}_{P(\ok)}$
, hence perfectoid. Since $\mc{M}_{\mbf{1}}^{(0)}/P^\prime$ is covered by opens of the form $(pr^{-1}(U) \times^{P(\ok)} P^\prime)/P^\prime$, $\mc{M}_{\mbf{1}}^{(0)}/P^\prime$ is perfectoid as desired.
This also shows that there is a cover of $\mc{M}_{P(K)}$ by open subsets of the form $(pr^{-1}(U) \times^{P(\ok)} P^\prime)/P^\prime \cong pr^{-1}(U)/P(\ok)$, that embed into $\mc{M}_{P(\ok)}^{(0)}$ and give sections of the natural projection map. 
\end{proof}

Since the Gross--Hopkins period map $\mc{M}_{\mbf{1}} \ra \Pro$ is $\GL_n(K)$-equivariant for the trivial action on the target, it factors over $\mc{M}_{\mbf{1}} \ra \mc{M}_{P(K)}$; we write
$$ \ol{\pi}_{GH} : \mc{M}_{P(K)} \ra \Pro $$
for this factorization. We get the following generalization of \cite[Proposition 3.15]{lud}, by exactly the same proof. 

\begin{proposition}\label{qc}
$\ol{\pi}_{GH}$ is quasicompact. As a consequence, $\mc{M}_{P(K)}$ is quasicompact. Moreover, $\mc{M}_{P(K)}$ is quasiseparated.
\end{proposition}

\begin{proof}
The proof that $\ol{\pi}_{GH}$ is quasicompact is identical to the proof of the special case \cite[Proposition 3.15]{lud} when $n=2$ and $K=\Qp$; we recall it briefly since the argument also proves that $\mc{M}_{P(K)}$ is quasiseparated. In short, since $\pi_{GH}$ has local sections, $\Pro$ is covered by quasicompact open subsets $V$ for which there exists an open $U \sub \mc{M}^{(0)}$ such that $\pi_{GH}|_U$ is an isomorphism onto $V$. By the argument in the proof of Theorem \ref{main},
$$ \ol{\pi}_{GH}^{-1}(V) \cong (pr^{-1}(U)\times^{P(\ok)}P^{\prime})/P^{\prime} \cong pr^{-1}(U)/P(\ok), $$
which is quasicompact, so $\ol{\pi}_{GH}$ is quasicompact (and hence so is $\mc{M}_{P(K)}$ since $\Pro$ is quasicompact). To show that $\mc{M}_{P(K)}$ is quasiseparated we first show that $\ol{\pi}_{GH}^{-1}(V)$ is qcqs. To see this, note that $pr^{-1}(U)$ is an inverse limit of qcqs spaces, hence qcqs and therefore a spectral space. It then follows that the quotient $\ol{\pi}_{GH}^{-1}(V) \cong pr^{-1}(U)/P(\ok)$ is a spectral space by \cite[Lemma 3.2.3]{bfh}, so in particular qcqs. The intersection of two such subsets of $\mc{M}_{P(K)}$ is also quasicompact ($\ol{\pi}_{GH}^{-1}(V_1) \cap \ol{\pi}_{GH}^{-1}(V_2) = \ol{\pi}_{GH}^{-1}(V_1 \cap V_2)$), so $\mc{M}_{P(\ok)}$ is quasiseparated by \cite[VI, Corollaire 1.17]{sga4}.
\end{proof}

Thus we have shown that $|\mc{M}_{P(K)}|$ is a spectral space. We will also need the fact that it has Krull dimension $n-1$, i.e., that the supremum of all lengths $k$ of generalizations $x_0\prec \dots \prec x_k$ is equal to $n-1$. To make the proof transparent, we record a few simple observations on Krull dimensions.

\begin{lemma}\label{dimensionobservations}
Let $X$ and $Y$ be locally spectral spaces.
\begin{enumerate}
\item If $X$ is a cofiltered inverse limit $\varprojlim_i X_i$ of locally spectral spaces, then $\dim X \leq \sup_i \dim X_i$.

\smallskip

\item If $f : X \ra Y$ is a surjective and generalizing continuous map, then $\dim X \geq \dim Y$.
\end{enumerate}
\end{lemma}

\begin{proof}
We start with (1). Write $q_i : X \ra X_i$ for the natural map. If $x_0 \prec\dots \prec x_n$ is a chain of distinct generalizations in $X$, then $q_i(x_0) \preceq\dots \preceq q_i(x_n)$ is a chain of generalizations in $X_i$ for any $i$, and the $q_i(x_j)$ will be distinct for some $i$. This proves (1). 

\medskip
For (2), let $y_0 \prec \dots \prec y_m$ be a chain of distinct generalizations in $Y$. Then we can lift $y_0$ to a point $x_0\in X$ by surjectivity of $f$, and then successively lift the $y_i$, $i\geq 2$, using that $f$ is generalizing, to obtain a chain $x_0 \prec \dots \prec x_m$ in $X$, proving (2).
\end{proof}

\begin{proposition}
$|\mc{M}_{P(K)}|$ is a spectral space of Krull dimension $n-1$.
\end{proposition}

\begin{proof}
Since $\mc{M}_{\mbf{1}}$ is an inverse limit of rigid analytic varieties of dimension $n-1$, it has dimension $\leq n-1$ by Lemma \ref{dimensionobservations}(1). Applying Lemma \ref{dimensionobservations}(2) to the surjective and generalizing\footnote{Any map of analytic adic spaces is generalizing.} maps $\mc{M}_{\mbf{1}} \ra \mc{M}$ and $\mc{M}_{\mbf{1}} \ra \mc{M}_{P(K)}$, we see that $\dim \mc{M}_{\mbf{1}} = n-1$ and that $\dim \mc{M}_{P(K)} \leq n-1$. To prove equality, one may argue exactly as at the end of the proof of \cite[Lemma 3.2.3]{bfh}, using that $\mc{M}_{P(K)}$ is the quotient of $\mc{M}_{\mbf{1}}$ by $P(K)$ in the category $\V$.
\end{proof}

We will end this section by showing that $\mc{M}_\mbf{1}$ is a $P(K)$-torsor over $\mc{M}_{P(K)}$. For this, we first record two lemmas concerning the pushouts defined in \cite[\S 2]{lud}. 

\begin{lemma}\label{help3}
Let $G$ be a locally profinite group and let $H \sub G$ be a compact open subgroup. Assume that $H$ acts on a perfectoid space $X$, that $G$ acts on a perfectoid space $Y$ and that we have an $H$-invariant map of perfectoid spaces $X \ra Y$. Then there is a natural $G$-invariant map $X \times^H G \ra Y$, and if $Z \ra Y$ is a map of perfectoid spaces then the natural map $(X\times_Y Z) \times^H G \ra (X \times^H G) \times_Y Z$ is a $G$-equivariant isomorphism.
\end{lemma}

\begin{proof}
The existence of $X \times^H G \ra Y$ is \cite[Lemma 2.16]{lud}. For the compatibility with fibre products we note that there is indeed a natural map $(X\times_Y Z) \times \ul{G} \ra (X \times^H G) \times_Y Z$ given by $(x,z,g) \ra (x,g,z)$. It is easily checked to be both $H$-invariant for the action $(x,z,g).h = (xh,z,h^{-1}g)$ on $(X\times_Y Z) \times \ul{G}$, and $G$-equivariant for the action given by acting by right translation on the $G$-factor on the target and source. These actions commute and so induce the natural $G$-equivariant map $(X\times_Y Z) \times^H G \ra (X \times^H G) \times_Y Z$. To see that it is an isomorphism, use the description of the pushout from \cite[Proposition 2.15]{lud} and the fact that disjoint unions commute with fibre products.
\end{proof}

\begin{lemma}\label{help1}
Let $G$ be a locally profinite group and let $H \sub G$ be a compact open subgroup. If $X \ra Y$ is an $H$-torsor of perfectoid spaces, then $X \times^H G \ra Y$ is a $G$-torsor of perfectoid spaces.
\end{lemma}

\begin{proof}
$X \ra Y$ is a v-cover, so it suffices to show that $(X \times^H G) \times_Y X \cong X \times \ul{G}$, $G$-equivariantly. Using Lemma \ref{help3} and the fact that $X \ra Y$ is an $H$-torsor we see that
$$ (X \times^H G) \times_Y X \cong (X \times_Y X) \times^H G \cong (X \times \ul{H}) \times^H G \cong X \times \ul{G} $$
and one checks that these isomorphisms are all $G$-equivariant.
\end{proof}

Using these we can now prove that $\mc{M}_\mbf{1}$ is a $P(K)$-torsor over $\mc{M}_{P(K)}$.

\begin{proposition}\label{ptorsor}
$\mc{M}_\mbf{1}$ is a $P(K)$-torsor over $\mc{M}_{P(K)}$.
\end{proposition}

\begin{proof}
The statement is local on $\mc{M}_{P(K)}$, so we may restrict to the types of open subsets
$$ pr^{-1}(U) \times^{P(\ok)} P(K)/P(K) \cong pr^{-1}(U)/P(\ok)$$
used in the proof of Theorem \ref{main}, which have preimage $pr^{-1}(U) \times^{P(\ok)} P(K)$ in $\mc{M}_\mbf{1}$. Then, by Lemma \ref{help1}, we see that it suffices to show that $pr^{-1}(U) \ra pr^{-1}(U)/P(\ok)$ is a $P(\ok)$-torsor, but this follows by construction (arguing as in, or using, Lemma \ref{torsor}). 
\end{proof}

As a consequence, we note that $\mc{M}_H \ra \mc{M}_{P(K)}$ is (separated and) \'etale for any open subgroup $H \sub P(\ok)$, by \cite[Lemma 10.13]{dia}.

\section{Application to Scholze's functor} \label{sec:app}
\subsection{Recollections}
We recall some results of \cite{plt}. Let $D/K$ be a central division algebra of invariant $1/n$. 
For a smooth admissible representation $\pi$ of $\GL_n(K)$ on a $\mathbb{F}_p$-vector space, Scholze defines a sheaf $\mc{F}	_{\pi}$  on $(\mathbb{P}^{n-1})_{\et}$ by 
\[\mc{F}_{\pi}(U) =  \mathrm{Map}_{\mathrm{cont},\GL_n(K)}(|U\times_{\mathbb{P}^{n-1}} \mc{M}_{1}|, \pi)\]
(where $U\rightarrow \mathbb{P}^{n-1}$ is an \'etale map) and shows that the cohomology groups 
\[\mathcal{S}^i(\pi):= H_\et^i(\mathbb{P}^{n-1},\mc{F}_\pi), i \geq 0,\]
are admissible $D^\times$-representations which carry an action of $\Gal(\overline{K}/K)$ and vanish in degree $i > 2(n-1)$ (\cite[Theorem 1.1]{plt}). The main result of this section is Theorem \ref{vanish}, which shows that in fact $\mc{S}^i(\pi)=0$ for $i>n-1$ whenever $\pi$ is induced from the parabolic $P$.

\subsection{Some cohomological calculations}

In preparation for Theorem \ref{vanish}, we carry out some auxiliary calculations. We begin with some remarks about the geometric fibres of $\ol{\pi}_{GH}$. Let $\ol{x} : \Spa(E,E^+) \ra \Pro$ be a geometric point. We define the fibre $(\mc{M}_{P(K)})_{\ol{x}}$ as the fibre product
\[
 (\mc{M}_{P(K)})_{\ol{x}} := \mc{M}_{P(K)} \times_{(\Pro)^\lozenge} \Spa(E,E^+) 
\]
in the category of diamonds. Since $\mc{M}_{\mbf{1}} \ra \Pro$ is a $\GL_n(K)$-torsor and $\mc{M}_\mbf{1} \ra \mc{M}_{P(K)}$ is a $P(K)$-torsor (by Proposition \ref{ptorsor}), the geometric fibres of $\ol{\pi}_{GH}$ are profinite sets
$$ (\mc{M}_{P(K)})_{\ol{x}} \cong \ol{x} \times \ul{S}, $$
with $S=\GL_n(K)/P(K) = \GL_n(\mc{O}_K)/P(\mc{O}_K)$ (we refer to e.g. \cite[Proposition 2.10]{lud} for a definition of the notation $\ol{x} \times \ul{S}$; see also \cite[Example 11.12]{dia}).

\begin{lemma}\label{aux2} Let $\mathcal{F}$ be a sheaf of abelian groups on $(\mc{M}_{P(K)})_{\et}$. Then 
\[H^i_{\et}(\mc{M}_{P(K)}, \mathcal{F}) = H^i_{\et}(\mathbb{P}^{n-1}, \overline{\pi}_{\GH,*}\mathcal{F})\]
for all $i\geq 0$. 
\end{lemma}
\begin{proof} This is proved exactly as \cite[Proposition 4.4]{lud}, using Proposition \ref{qc} and the fact that the geometric fibres $(\mc{M}_{P(K)})_{\overline{x}}$ are profinite sets over $\ol{x}$.
\end{proof}

\begin{proposition}\label{aux1} Let $\mathcal{F}$ be a sheaf of $\mathbb{F}_p$-vector spaces on $(\mc{M}_{P(K)})_{\et}$. We have an isomorphism of sheaves on $(\mathbb{P}^{n-1})_{\et}$
\[(\overline{\pi}_{\GH,*}\mathcal{F})\otimes \mc{O}^+_{\mathbb{P}^{n-1}}/p \cong \overline{\pi}_{\GH,*}(\mathcal{F}\otimes \mc{O}^+_{\mc{M}_{P(K)}}/p).\]
\end{proposition}
\begin{proof}
We give a slightly different proof than in \cite[Lemma 4.5]{lud}. There is a natural map
\[(\overline{\pi}_{\GH,*}\mathcal{F})\otimes \mc{O}^+_{\mathbb{P}^{n-1}}/p \rightarrow \overline{\pi}_{\GH,*}(\mathcal{F}\otimes \mc{O}^+_{\mc{M}_{P(K)}}/p),\]
so we can check the assertion on stalks at geometric points. 
For that let $\overline{x}=\Spa(E,E^+)$ be a geometric point of $\mathbb{P}^{n-1}$.
On the one hand 
\begin{eqnarray*}((\overline{\pi}_{\GH,*}\mathcal{F}) \otimes \mc{O}^+_{\mathbb{P}^{n-1}}/p)_{\overline{x}} &\cong &(\overline{\pi}_{\GH,*}\mathcal{F})_{\overline{x}} \otimes (\mc{O}^+_{\mathbb{P}^{n-1}}/p)_{\overline{x}} \\
&\cong & H^0_{\et}(\left(\mc{M}_{P(K)}\right)_{\overline{x}},\mathcal{F})\otimes E^+/p,
\end{eqnarray*}
by \cite[Proposition 2.2.4]{arizona}. On the other hand, applying that same proposition we get
\[\overline{\pi}_{\GH,*}(\mathcal{F} \otimes \mc{O}^+_{\mc{M}_{P(K)}}/p)_{\overline{x}} \cong H_{\et}^0\left(\left(\mc{M}_{P(K)})\right)_{\overline{x}}, \mathcal{F} \otimes \mc{O}^+_{\mc{M}_{P(K)}}/p\right).
\]
We have $\left(\mc{M}_{P(K)}\right)_{\overline{x}} \cong \overline{x} \times \underline{S}$ with $S$ a profinite set, so we are left to show that the natural map
\begin{equation}\label{iso} H^0_{\et}(\ol{x}\times \ul{S},\mc{F})\otimes E^+/p \ra H_{\et}^0\left(\ol{x}\times \ul{S}, \mc{F} \otimes \mc{O}^+_{\ol{x}\times \ul{S}}/p\right) 
\end{equation}
is an isomorphism. For that write $S$ as an inverse limit $S=\varprojlim S_i$ of finite sets $S_i$ and denote by $q_i: \ol{x}\times \ul{S} \ra \ol{x} \times S_i$ the natural projection morphism. By \cite[VI, 8.3.13]{sga4}, any sheaf on $\ol{x} \times \ul{S}$ can be written as a filtered colimit $\varinjlim_{j \in J} \mc{F}^j$ of sheaves $\mc{F}^j$ that arise as the inverse image of a system of sheaves $\mc{F}^j_i$ on the spaces $(\ol{x}\times S_i)_{\et}$. The topos $(\ol{x}\times \ul{S})_{\et}$ is coherent, so (\'etale) cohomology commutes with direct limits. As tensor products also commute with direct limits it suffices to prove (\ref{iso}) for sheaves of the form $\mc{F}\cong \varinjlim q^{-1}_i \mc{F}_i$ for some sheaves $\mc{F}_i$ on $(\ol{x} \times S_i)_{\et}$.

Note that $\mc{O}^+_{\ol{x}\times \ul{S}}/p \cong \varinjlim q^{-1}_i(\mc{O}^+_{\ol{x}\times S_i}/p)$\footnote{One checks this by calculating sections on the basis for the topology consisting of open affinoid perfectoids $U$ of the form $U = \varprojlim U_i$, for open affinoid perfectoid $U_i \subset \ol{x}\times \ul{S}$, using the fact that those don't have any higher \'etale cohomology.}. 
Using \cite[Theorem 2.4.7]{sw} we see that we can rewrite (\ref{iso}) as
\begin{equation*} 
\varinjlim H^0_{\et}(\ol{x}\times S_i,\mc{F}_i)\otimes E^+/p \ra \varinjlim H_{\et}^0\left(\ol{x}\times S_i, \mc{F}_i \otimes \mc{O}^+_{\ol{x}\times S_i}/p\right), 
\end{equation*}
and we see this map is indeed an isomorphism as the spaces $\ol{x}\times S_i$ are just finite disjoint unions of geometric points with the same underlying affinoid field $(E,E^+)$. 
\end{proof}

Next, let $\sigma$ be a smooth admissible representation of $P(K)$. Define a sheaf $\mathcal{F}_\sigma$ on $(\mc{M}_{P(K)})_{\et}$ by 
\[\mathcal{F}_\sigma(U) = \mathrm{Map}_{\mathrm{cont},P(K)}(|U\times_{\mc{M}_{P(K)}} \mc{M}_{1}|, \sigma)\]
for $U\rightarrow \mc{M}_{P(K)}$ \'etale. Similarly, if $\tau$ is a smooth admissible representation of $P(\ok)$, then we may define a sheaf $\mc{F}_\tau$ on $\mc{M}_{P(\ok)}$ by 
\[\mathcal{F}_\tau(V) = \mathrm{Map}_{\mathrm{cont},P(\ok)}(|V\times_{\mc{M}_{P(\ok)}} \mc{M}_{1}|, \tau),\]
where $V \ra \mc{M}_{P(\ok)}$ is \'etale. Since the natural map $q : \mc{M}_{P(\ok)} \ra \mc{M}_{P(K)}$ is \'etale, we have a natural map
$$ q^{-1}\mc{F}_{\sigma} \ra \mc{F}_{\sigma|_{P(\ok)}} $$
for any smooth admissible $P(K)$-representation $\sigma$ and its restriction $\sigma|_{P(\ok)}$ to $P(\ok)$.

\begin{lemma}\label{restriction}
The natural map $ q^{-1}\mc{F}_{\sigma} \ra \mc{F}_{\sigma|_{P(\ok)}} $ is an isomorphism.
\end{lemma}

\begin{proof}
We may check on stalks, so let $\ol{x} \ra \mc{M}_{P(\ok)}$ be a geometric point. We may assume that $\ol{x} = \varprojlim_{U \ra \mc{M}_{P(\ok)}} U$, where the limit ranges over $U \ra \mc{M}_{P(\ok)}$ \'etale over which $\ol{x} \ra \mc{M}_{P(\ok)}$ factors (see \cite[\S 2.2]{arizona}). We then have
\begin{eqnarray*}(q^{-1}\mc{F}_\sigma)_{\ol{x}} &= &\varinjlim_U {\rm Map}_{\mathrm{cont}, P(K)} (|U \times_{\mc{M}_{P(K)}}\mc{M}_\mbf{1}|, \sigma) \\
&\cong & {\rm Map}_{\mathrm{cont}, P(K)} (\varprojlim_U |U \times_{\mc{M}_{P(K)}}\mc{M}_\mbf{1}|, \sigma) \\
&\cong & {\rm Map}_{\mathrm{cont}, P(K)} ( |\ol{x} \times_{\mc{M}_{P(K)}}\mc{M}_\mbf{1}|, \sigma) \\
&\cong & {\rm Map}_{\mathrm{cont}, P(K)} ( |\ol{x}| \times P(K), \sigma) \cong \sigma
\end{eqnarray*}
upon choosing an element in $P(K)$; here we have used Proposition \ref{ptorsor} to get the second to last isomorphism. We similarly have $(\mc{F}_{\sigma|_{P(\ok)}})_{\ol{x}}\cong \sigma$ (choosing the same element and the map $(q^{-1}\mc{F}_{\sigma})_{\ol{x}} \ra (\mc{F}_{\sigma|_{P(\ok)}})_{\ol{x}}$ corresponds to the identity $\sigma \ra \sigma$, and is therefore an isomorphism.
\end{proof}

\begin{proposition}\label{etan} Let $\lambda: (\mc{M}_{P(K)})_{\et} \rightarrow |\mc{M}_{P(K)}|$ denote the natural morphism of sites.  
 For any admissible smooth representation $\sigma$ of $P(K)$ we have an almost isomorphism
\[H^i_{\et}(\mc{M}_{P(K)}, \mathcal{F}_\sigma \otimes \mc{O}^+/p)\cong^a H^i(|\mc{M}_{P(K)}|, \lambda_*(\mathcal{F}_\sigma \otimes \mc{O}^+/p)).\]
\end{proposition}
\begin{proof} (Cf.~proof of \cite[Theorem 3.2]{plt} on p.~18 for a similar argument.) We show that $(R^q\lambda_*(\mathcal{F}_\sigma \otimes \mc{O}^+/p))=^a 0$ for all $q>0$ . For this we calculate the stalks. Let $x:\Spa(K,K^+)\rightarrow \mc{M}_{P(K)}$ be a point.	
Then, by definition,
\[(R^q\lambda_*(\mc{F}_\sigma \otimes \oo^+/p))_x = \varinjlim_{x \in U} H^q_{\et}(U, \mathcal{F}_\sigma \otimes \mc{O}^+/p),\]
 where the direct limit runs over all open $U \sub \mc{M}_{P(K)}$ containing $x$, and we can restrict it to those $U$ which are affinoid perfectoid. Since $\mc{M}^{(0)}_{P(\mc{O}_K)}\ra \mc{M}_{P(K)}$ has local sections, we may furthermore assume that $U$ is (isomorphic to) an open subset of $\mc{M}^{(0)}_{P(\mc{O}_K)}$. On such a $U$, Lemma \ref{restriction} implies that $\mathcal{F}_\sigma \cong \varinjlim_V \mc{F}_{\sigma^V}$, where $V$ runs over the open normal subgroups of $P(\mc{O}_K)$. Then 
\[H^q_{\et}(U, \mc{F}_{\sigma} \otimes \mc{O}^+/p)\cong \varinjlim_V H^q_{\et}(U, \mc{F}_{\sigma^V} \otimes \mc{O}^+/p),\]
as the \'etale site of $U$ is coherent and direct limits commute with tensor products. But for any open normal subgroup $V \subset P(\mc{O}_K)$, the sheaf $\mc{F}_{\sigma^V}$ is a local system of finite rank, and therefore we have
\[H^q_{\et}(U, \mathcal{F}_{\sigma^V} \otimes \mc{O}^+/p) =^a 0\]
for all $q >0$, by \cite[Lemma 4.12]{pht}.
\end{proof}

\subsection{The vanishing result}
We now prove our vanishing result. 
\begin{theorem}\label{vanish} Let $P^*\subset \GL_n$ be a parabolic subgroup contained in $P$. Let $\sigma$ be a smooth admissible representation of $P^*(K)$. Let $\pi:= \mathrm{Ind}^{\GL_n(K)}_{P^*(K)}{\sigma}$ be the parabolic induction (which is a smooth admissible representation of $\GL_n(K)$). Then
\[\mathcal{S}^i(\pi)=0 \ \text{ for all } i> n-1.\]
\end{theorem}
\begin{proof} Transitivity of parabolic induction immediately implies that we can reduce to the case $P^*=P$. We then follow the proof of \cite[Theorem 4.6]{lud}.
It suffices to show that 
\[H^i_{\et}(\mathbb{P}^{n-1}, \mathcal{F}_\pi) \otimes \mathcal{O}^+/p \]
is almost zero for all $i > n-1$. 
We have isomorphisms 
\begin{eqnarray}
H^i_{\et}(\mathbb{P}^{n-1}, \mathcal{F}_\pi) \otimes \mathcal{O}^+/p &\cong^a& H^i_{\et}(\mathbb{P}^{n-1}, (\overline{\pi}_{\GH,*}\mathcal{F}_{\sigma}) \otimes \mathcal{O}^+/p )\\
																															& \cong & H^i_{\et}(\mathbb{P}^{n-1},\overline{\pi}_{\GH, *}(\mathcal{F}_\sigma \otimes \mathcal{O}^+/p )) \\
																															&\cong& H^i_{\et}(\mc{M}_{P(K)}, \mathcal{F}_\sigma \otimes \mathcal{O}^+/p),
\end{eqnarray}
where the first almost isomorphism follows from \cite[Theorem 3.2]{plt} and the fact that $\overline{\pi}_{\GH,*}\mathcal{F}_{\sigma} \cong \mathcal{F}_{\pi}$, which one proves just like \cite[Lemma 4.3]{lud}. The second isomorphism is Proposition \ref{aux1} above, the third is Lemma \ref{aux2}. 
By Proposition \ref{etan}, the \'etale cohomology group $H^i_{\et}(\mc{M}_{P(K)}, \mathcal{F}_\sigma \otimes \mathcal{O}^+/p)$ is almost isomorphic to the analytic cohomology group $H^i(|\mc{M}_{P(K)}|, \lambda_*(\mathcal{F}_\sigma \otimes \mathcal{O}^+/p))$. 

As we have seen in Section \ref{sec:LT}, $|\mc{M}_{P(K)}|$ is a spectral space of Krull dimension $n-1$, therefore by \cite[Theorem 4.5]{dim}
\[H^i(|\mc{M}_{P(K)}|, \lambda_*(\mathcal{F}_\sigma \otimes \mathcal{O}^+/p))= 0\]
for all $i> n-1$. 
\end{proof}

\begin{corollary}\label{weak} Let $\pi$ be a representation of $\GL_n(K)$ that appears as a quotient of a parabolically induced representation $\mathrm{Ind}_{P^*}^{\GL_n(K)}\sigma$, for some parabolic subgroup $P^*\subset P$. Then 
\[\mathcal{S}^{2(n-1)}(\pi) = H^{2(n-1)}_{\et}(\mathbb{P}^{n-1}, \mathcal{F}_{\pi})=0.\] 
\end{corollary}
\begin{proof} This follows from exactness of the functor $\pi \mapsto \mathcal{F}_{\pi}$, Theorem \ref{vanish} and the long exact sequence in cohomology.
\end{proof}

\begin{remark} We finish with some remarks on our results.

\begin{enumerate}

\item The bound on cohomological vanishing in Corollary \ref{weak} (combined with \cite[Theorem 3.2]{plt}) is sharp in general, and for general subquotients of representations induced from $P(K)$ the bound from \cite[Theorem 3.2]{plt} cannot be improved. To see these two things (simultaneously), consider the trivial representation $\mbf{1}$ and the exact sequence
$$ 0 \ra \mbf{1} \ra \sigma = {\rm Ind}_{P(K)}^{\GL_n(K)}\mbf{1} \ra Q \ra 0, $$
where $Q$ is simply defined to be the quotient; $Q$ is then a generalized Steinberg representation and known to be irreducible (and admissible). From this we get an exact sequence of \'etale sheaves
\[
 0 \ra \mc{F}_\mbf{1} \ra \mc{F}_\sigma \ra \mc{F}_Q \ra 0 
\]
on $\Pro$. Note that $\mc{F}_\mbf{1}$ is the trivial local system on $\Pro$, so $\mc{S}^{2(n-1)}(\mbf{1})\neq 0$; this shows the second point. The long exact sequence then shows (using Theorem \ref{vanish}) that $\mc{S}^{2n-3}(Q)$ surjects onto $\mc{S}^{2(n-1)}(\mbf{1})$, so $\mc{S}^{2n-3}(Q)\neq 0$ as well, proving the first point. Thus, the vanishing result of Theorem \ref{vanish} does not hold in general for admissible representations not induced from $P(K)$ when $n\geq 3$, even for `reasonable' representations like generalized Steinberg representations (which are irreducible, and infinite-dimensional). 

\smallskip

Note also that, as a consequence of Corollary \ref{weak}, the trivial representation cannot be written as a quotient of a representation induced from $P(K)$. We thank Florian Herzig for informing us that this is well known, and is easily proved using the adjunction formula between parabolic induction and Emerton's ordinary parts functor.

\medskip

\item Elaborating further on the previous item, it seems interesting to understand in which degrees $\mc{S}^i(\pi)$ vanish for different classes of admissible representations $\pi$.  A natural question is whether the analogue of Theorem \ref{vanish} holds for other maximal (standard) parabolics $Q \neq P$. By Theorem \ref{0.5}, the quotient $\mc{M}_{Q(K)}$ is not a perfectoid space, and so the method for proving Theorem \ref{0.5} breaks down. One could ask whether the vanishing theorem could still be salvaged by geometric methods (such as in \cite{arizona}, where a vanishing result is proven in a situation where the space in question is not perfectoid), but we currently see no way of doing this (in particular, we see no way of adapting the method of \cite{arizona}). 

\medskip

\item It is also natural to ask about vanishing below the middle degree, but here things seem to be much more unclear. For $\mc{S}^0$, we have $\mc{S}^0(\pi) = \mc{S}^0(\pi^{{\rm SL}_n(K)}) $ by \cite[Proposition 4.7]{plt}, so e.g. when $\pi$ is irreducible and infinite-dimensional we know that $\mc{S}^0(\pi)=0$. When $n=2$ the middle degree is $1$, so in this case (for arbitrary $K$), we can say that $\mc{S}^i(\pi)$ is concentrated in degree $1$ for irreducible $\pi= {\rm Ind}_{P(K)}^{\GL_n(K)}\sigma$.

\smallskip

As a referee emphasized, it seems natural to ask if there is some form of Poincar\'e duality for $\mc{S}^i$ that could relate the degrees below the middle to degrees above the middle. The following remark is due to David Hansen; we thank him for allowing us to include it here. First, note that such a duality would presumably require the notion of a `dual' local system $\mc{F}_\pi^\vee$, and presumably $\mc{F}_\pi^\vee = \mc{F}_{\pi^\vee}$ where $\pi^\vee$ is the `dual' representation of $\pi$. However, such a formulation seems much too naive, as duality is much more subtle in characteristic $p$ than in characteristic $\neq p$. Kohlhaase \cite{kohl} has defined a derived duality functor on the derived category of smooth $G(\Qp)$-representations in characteristic $p$ (where $G/\Qp$ is, momentarily, an arbitrary connected reductive group), and one could ask how a derived version $R\mc{S}$ of Scholze's functor interacts with Kohlhaase's duality functor for $GL_n(K)$. It seems, however, that this interaction will also involve Kohlhaase's duality functor for $D^\times$, and it is unclear if such considerations can be used to reduce the study of low degrees to high degrees.

\medskip

\item We end by remarking that Pa\v{s}k\={u}nas \cite{pask} has used the results of \cite{lud} to show a non-vanishing result in degree one for (a version of) Scholze's functor for Banach space representations of $\GL_2(\Qp)$ corresponding to residually reducible two-dimensional representations of ${\rm Gal}(\ol{\mathbb{Q}}_p /\Qp)$ via the $p$-adic local Langlands correspondence (we refer to \cite{pask} for precise statements). It would be interesting to see similar consequences for $\GL_2(K)$, where $K/\Q_p$ is arbitrary. However, Pa\v{s}k\={u}nas informs us that our results would not be sufficient even assuming a $p$-adic local Langlands correspondence for $\GL_2(K)$, as it is expected that supersingular representations will contribute to the Banach space representation corresponding to reducible two-dimensional representations of ${\rm Gal}(\ol{K} / K)$ when $K\neq \Qp$. Nevertheless, we hope that our results will be useful for the further study of Scholze's functor.

\end{enumerate}
\end{remark}

\appendix
\counterwithin{theorem}{subsection}
\section{Perfectoid quotients of the Lubin--Tate tower, revisited \\ By David Hansen}

\subsection{Statement of results}

As in the main text of the paper, fix a finite extension $K/\Qp$ with residue field $k\cong \F_q$. Let $\breve{K}$ be the completed maximal unramified extension of $K$, and fix some complete algebraically closed extension $C/\breve{K}$. For simplicity, we will write $\Spa R := \Spa(R,R^\circ)$ and  $\Spd R := \Spa(R,R^\circ)^\lozenge$ whenever $R$ is a Huber ring over $\Zp$. Here $\Spa(R,R^\circ)^\lozenge$ is the v-sheaf defined in \cite[Definition 18.1]{berkeley}, generalizing the diamondification functor. Moreover, a perfectoid space $S$ over $R$ (as above) always means a perfectoid space $S$ with a map $(R,R^\circ) \ra (\oo_S(S),\oo_S^+(S))$.

\medskip

Let $P_{n-d,d}\subset\GL_{n}$ denote the usual block upper-triangular maximal parabolic with upper left (resp. lower right) diagonal block sizes $n-d$ (resp. $d$), and set $P=P_{n-1,1}$ .
For $U\subset \GL_n(K)$ any open compact subgroup, let $\mc{M}_U$
denote the base change along $\breve{K}\to C$ of the generic fiber
of the Lubin--Tate space with $U$-level structure. By the results
in \cite{sw}, there is a unique perfectoid space $\mc{M}_{\mbf{1}}$ over $C$ with a $\GL_n(K)$-action such that
$\mc{M}_{\mbf{1}} \sim \varprojlim_U \mc{M}_U$, in the notation of \cite[\S 2.4]{sw}.

\medskip

In this appendix we study the sheaf quotient $\mc{M}_{\mbf{1}}/\ul{P(K)}$ (with $\mc{M}_{\mbf{1}}$ viewed as a diamond), and variants for other parabolics, by purely local methods. In particular, when $n=2$, we give a direct proof that $\mc{M}_{\mbf{1}} / \ul{P(K)}$ is a perfectoid space, avoiding any reference to the $p$-adic geometry of Shimura varieties. Our main tool is a $p$-adic Hodge-theoretic description of $\mc{M}_{\mbf{1}}$ in terms of vector bundles on the Fargues--Fontaine curve, due to Weinstein and Scholze--Weinstein. 

\medskip

Our first result is a moduli-theoretic description of these quotients. For this, recall that for any characteristic $p$ perfectoid space $S/k$ there is an associated (adic, relative) Fargues--Fontaine curve $\mc{X}_{S}=\mc{X}_{S,K}$ defined functorially in $S$, cf. \cite[Ch. 8]{ked-liu}. Moreover for any reduced rational number $\lambda=d/r$, this curve comes with a canonical vector bundle $\mc{O}(\lambda)$ of rank $r$ and degree $d$, defined functorially in $S$.

\begin{theorem}\label{0.1}
As a functor on perfectoid spaces over $C$, the diamond
quotient $\mc{M}_{\mbf{1}}/\ul{P(K)}$ is canonically
equivalent to the functor sending any $S\to\Spa C$ to the
set of sub-vector bundles $\mc{E} \subset \mc{O}(1/n)$ over
$\mc{X}_{S^{\flat}}$ such that after pullback along any geometric
point $\ol{x}\to S^{\flat}$, the map $\mc{E}_{\ol{x}} \to \mc{O}(1/n)_{\ol{x}}$
is injective and $\mc{E}_{\ol{x}} \simeq \mc{O}^{n-1}$.
\end{theorem}

Curiously, this description shows that $\mc{M}_{\mbf{1}}/\ul{P(K)}$
is the base change along $\Spd C \to \Spd k$ of a naturally defined functor on all perfectoid spaces over $k$. We also note that, by Proposition \ref{ptorsor}, the diamond quotient $\mc{M}_{\mbf{1}}/\ul{P(K)}$ coincides a posteriori with the perfectoid space $\mc{M}_{P(K)}$, so there is little harm in switching between these points of view.

\begin{corollary}\label{0.2}
As a diamond over $\Spd C$, $\mc{M}_{\mbf{1}}/\ul{P(K)}$
is proper and $\ell$-cohomologically smooth for any $\ell\neq p$.
\end{corollary}

By \cite[Theorem 1.12]{dia}, this implies very strong finiteness
properties for the $\F_{\ell}$-\'etale cohomology of $\mc{M}_{\mbf{1}}/ \ul{P(K)}$. On the other hand, $\mc{M}_{\mbf{1}}/ \ul{P(K)}$ is a perfectoid space by Theorem \ref{A}. Thus $\mc{M}_{\mbf{1}}/ \ul{P(K)}$ is an example of an interesting and naturally occurring perfectoid space with reasonable \'etale cohomology.

\medskip

When $n=2$, the description of $\mc{M}_{\mbf{1}}/ \ul{P(K)}$ can be unwound further.

\begin{theorem}\label{0.3}
If $n=2$, then $\mc{M}_{\mbf{1}}/ \ul{P(K)}$ can be
presented as the quotient
\[
\left(\Spd L^{\flat} \times_{\Spd \F_{q^2}} \Spd C \right)/(\varphi\times\mathrm{id})^{\Z}
\]
for a certain perfectoid field $L^{\flat}/\F_{q^2}$, where
$\varphi$ is the $q^2$-power Frobenius on $L^{\flat}$.
\end{theorem}

Combining this description with some formalism of diamonds, we obtain a purely local proof of (a generalization of) the main result of \cite{lud}, avoiding the global $p$-adic geometry of modular curves.

\begin{corollary}\label{0.4}
When $n=2$, the quotient $\mc{M}_{\mbf{1}}/ \ul{P(K)}$
is a perfectoid space over $C$.
\end{corollary}

In fact, it turns out that our methods give some information about
the more general quotients $\mc{M}_{\mbf{1}}/ \ul{P_{n-d,d}(K)}$. In particular, we prove the following result, which shows that Theorem \ref{A} in the paper is essentially optimal.

\begin{theorem}\label{0.5}
As a diamond over $\Spd C$, $\mc{M}_{\mbf{1}}/\ul{P_{n-d,d}(K)}$
is proper and $\ell$-cohomologically smooth for any $1\leq d<n$.
However, when $d>1$, this quotient is \emph{never }a perfectoid space.
\end{theorem}

Again, we deduce this from a moduli-theoretic description of these
more general quotients in terms of vector bundles on the curve, which recovers Theorem \ref{0.1} when $d=1$. When $d>1$ this description is slightly more complicated, cf. Proposition \ref{0.10} below.

\medskip

It's unclear to me how far these ideas can be extended beyond the
specific case of the Lubin--Tate tower. As an illustrative example,
let $\mc{N}_{\infty}$ be the infinite-level perfectoid space
over $C$ associated with the Rapoport-Zink tower for
an isoclinic $\varpi$-divisible $\mc{O}_{K}$-module of height
$5$ and dimension $2$. There is a natural action of $\GL_5(K)$
on $\mc{N}_{\infty}$, and one can check (by adapting the arguments
below) that the quotients $\mc{N}_{\infty}/ \ul{P_{i,5-i}(K)}$
are $\ell$-cohomologically smooth over $\Spd C$ for $i\in\{1,2,4\}$.
However, for $i=3$, the method breaks down, and I don't know whether
the quotient is smooth in that case.

\subsubsection*{Acknowledgments.}

I'm very grateful to Christian Johansson and Judith Ludwig for their invitation to write this appendix, and for some very interesting conversations about this circle of ideas. This appendix grew out of the (re)proof of Corollary \ref{0.4} given below, and I'd like to thank Jared Weinstein for some stimulating initial conversations around the question of whether this result could be proved by purely local methods. I would also like to thank an anonymous referee for useful comments and corrections. 

\subsection{Preliminaries}

For any perfectoid space $S/k$, we write $\mc{X}_{S}=\mc{X}_{S,K}$ for the associated relative Fargues--Fontaine curve, regarded as an
adic space over $K$. If $S=T^{\flat}$ arises as the tilt of some perfectoid space $T/K$, $\mc{X}_{T^{\flat}}$ comes
equipped with a canonical closed immersion $\iota: T \to \mc{X}_{T^{\flat}}$. Aside from the original reference \cite{ff}, some relevant background on the curve is given in \cite[Ch. 8]{ked-liu} and \cite[\S 2.3]{han-mod}. One might also look at \cite[\S 3.2-3.3]{cs} or at some portions of \cite{bfh}.

\medskip

We say that $S$ is a \emph{point} if $S=\Spa(L,L^{+})$ where
$L$ is a perfectoid field and $L^{+}\subset L$ is an open valuation
subring consisting of powerbounded elements. Moreover, we say $S$
is a \emph{rank one point} if $L^{+}=L^{\circ}$. 

\medskip

Now, when $S$ is a rank one point, $\mc{X}_{S}$ is a Noetherian
adic space of dimension one \cite{ked-noeth}, with a good theory of slopes and Harder--Narasimhan filtrations. Moreover, it is reduced and all of its local rings are fields or discrete valuation rings, so any coherent $\mc{O}_{\mc{X}_S}$-module $\mc{F}$ has a canonical filtration $0 \to \mc{F}_{\mathrm{tors}} \to \mc{F} \to \mc{F}_{\mathrm{free}} \to 0$ where $\mc{F}_{\mathrm{tors}}$ is a torsion coherent sheaf and $\mc{F}_{\mathrm{free}}$ is a vector bundle. In particular, given any vector bundle $\mc{F}$, any coherent subsheaf $\mc{E} \subset \mc{F}$ is also a vector bundle, and admits a canonical \emph{saturation $\mc{E}^{\mathrm{sat}} \subset \mc{F}$},
defined as the preimage of $(\mc{F}/ \mc{E})_{\mathrm{tors}} \subset\mc{F}/\mc{E}$ in $\mc{F}$. This is the minimal subbundle of $\mc{F}$ containing $\mc{E}$ such that $\mc{F}/ \mc{E}^{\mathrm{sat}}$
is also a vector bundle. Note also that $\mathrm{deg}(\mc{E}^{\mathrm{sat}}) = \mathrm{deg}(\mc{E}) + \mathrm{length}(\mc{F}/\mc{E})_{\mathrm{tors}}$. If $\mc{E}=\mc{E}^{\mathrm{sat}}$, we say that $\mc{E}$ is saturated. Moreover, all of these considerations extend to the case of $\mc{X}_{S}$ for $S$ a general point, not necessarily of rank one, thanks to the following general observation: for any point $S=\Spa(L,L^{+})$, pullback along the natural inclusion $\mc{X}_{\Spa(L,L^{\circ})} \to \mc{X}_{\Spa(L,L^{+})}$ induces an equivalence of categories on coherent $\mc{O}_{\mc{X}}$-modules.

\medskip

With these preparations, we can state a trivial lemma, which nevertheless is frequently very useful.

\begin{lemma}\label{0.6}
Suppose that $S$ is a point, and that $\mc{E} \to \mc{F}$
is an injective map of vector bundles on $\mc{X}_{S}$. Suppose
that the point $(\mathrm{rank}(\mc{E}),\mathrm{deg}(\mc{E})+1)$
lies above the Harder--Narasimhan polygon of $\mc{F}$. Then $\mc{E}$
is automatically saturated.

\medskip

In particular, if $\mc{F}$ is semistable and $\mathrm{\frac{\mathrm{deg}(\mc{E})+1}{\mathrm{rank}(\mc{E})}}>\mu(\mc{F})$, then $\mc{E}$ is automatically saturated.
\end{lemma}

\begin{proof}
Immediate from the fact that $\mathrm{HN}(\mc{E}^{\mathrm{sat}})$ lies on or
below $\mathrm{HN}(\mc{F})$, and from the fact that $\mathrm{deg}(\mc{E}^{\mathrm{sat}}) \geq \mathrm{deg}(\mc{E})+1$ if $\mc{E}$ is non-saturated.
\end{proof}

If $i : \mc{E} \to \mc{F}$ is any injective map of vector bundles
(or arbitrary $\mc{O}_{\mc{X}_{S}}$-modules) over a relative
curve $\mc{X}_{S}$, we say that $i$ is \emph{stably injective
}if it remains injective after base change along $\mc{X}_{T}\to\mc{X}_{S}$ for $T \to S$ any map of perfectoid spaces. This is equivalent to the a priori weaker condition that $\mc{E} \to\mc{F}$ remains injective after base change along $\mc{X}_{\ol{x}} \to \mc{X}_{S}$ for any geometric point $\ol{x}\to S$. This condition automatically holds if the quotient $\mc{F} / \mc{E}$ is a vector bundle, but in general it is weaker.

\begin{lemma}\label{0.7}
Let $\mc{E} \to \mc{F}$ be a stably injective map of vector bundles over a relative curve $\mc{X}_{S}$, such that $\mc{E} \simeq\mc{O}^{m}$ and $\mc{F} \simeq \mc{O}(1/n)$ at all geometric points of
$S$, for some fixed integers $m<n$. Then the quotient $\mc{F} / \mc{E}$ is a vector bundle, with $\mc{F}/\mc{E} \simeq \mc{O}(1/(n-m))$ at all geometric points of $S$.
\end{lemma}

\begin{proof}
When $S$ is a geometric point, the claim follows from the Fargues--Fontaine classification of vector bundles on $\mc{X}_{S}$. Indeed, consider an injective map $i : \mc{O}^{m} \to \mc{O}(1/n)$ with $m<n$. Since $\mc{O}(1/n)$ is stable, the previous Lemma implies that
$\mathrm{coker}\,i$ is automatically a vector bundle, which necessarily has rank $n-m$ and degree $1$. Moreover, all the Harder--Narasimhan slopes of $\mathrm{coker}\,i$ are $\geq 1/n$ (using the stability of $\mc{O}(1/n)$ again), so in particular, they are all positive, so the degree of $\mathrm{coker}\,i$ is bounded below by its number of distinct Harder--Narasimhan slopes. Thus $\mathrm{coker}\,i$ has a unique slope, which must be $1/(n-m)$, so $\mathrm{coker}\,i \simeq \mc{O}(1/(n-m))$.

\medskip

The result when $S$ is a (not necessarily geometric) point now follows by an easy descent (use that any injective map of Dedekind domains is flat). To check that $\mc{F} / \mc{E}$ is a vector bundle in general, note that our arguments so far imply that for any $S$
and any point $x \in |\mc{X}_{S}|$, the $k(x)$-rank of the fiber
$(\mc{F}/\mc{E})\otimes_{\mc{O}_{\mc{X}_{S}}}k(x)$
is $n-m$. Indeed, let $y \in |S|$ be the image of $x$ under the map
$|\mc{X}_{S}| \to |S|$; then formation of the $k(x)$-fiber factors
over the pullback of $\mc{E} \to \mc{F}$ along $\mc{X}_{y}=\mc{X}_{\Spa(k(y),k(y)^{+})} \to \mc{X}_{S}$, in the sense that $(\mc{F} / \mc{E})_{x} \cong (\mc{F}_{y} / \mc{E}_{y}) \otimes_{\mc{O}_{\mc{X}_{y}}}k(x)$. By our previous arguments, $\mc{F}_{y}/\mc{E}_{y}$ is a
vector bundle of rank $n-m$, so $\mathrm{rank}_{k(x)}(\mc{F}/\mc{E})\otimes_{\mc{O}_{\mc{X}_{S}}}k(x)$ is constant as a function of $x$. Since $\mc{X}_{S}$ is a stably uniform adic space, we then deduce from \cite[Proposition 2.8.4]{ked-liu} that $\mc{F}/\mc{E}$ is a finite locally free $\mc{O}_{\mc{X}_{S}}$-module.
\end{proof}

\begin{remark}
The argument in the preceding proof shows more generally that if $i:\mc{E} \to \mc{F}$ is any stably injective map of vector bundles over a relative curve $\mc{X}_{S}$ such that $\mathrm{coker}\,i_{\ol{x}}$ is torsion-free after pullback along any geometric point $\ol{x}\to S$, then $\mathrm{coker}\,i$ is a vector bundle.
\end{remark}

\subsection{General results}

In this section we prove Theorems \ref{0.1} and \ref{0.5}, and Corollary \ref{0.2}. Our starting point is the following result of Scholze--Weinstein, which is a special case of \cite[Cor. 23.2.2 and Cor. 24.3.5]{berkeley} (cf. also \cite{sw}).

\begin{proposition}\label{0.8}
As a functor on perfectoid spaces over $C$, $\mc{M}_{\mbf{1}}$ is canonically identified with the functor sending any $S \to \Spa C$ to the set of stably injective maps $\alpha : \mc{O}^{n} \to \mc{O}(1/n)$ over $\mc{X}_{S^{\flat}}$ such that $\mathrm{coker}\,\alpha \simeq \iota_{\ast}W$ for some rank one projective $\mc{O}_{S}$-module $W$.
\end{proposition}

Next, we note that for a closed subgroup $H \subset \GL_n(K)$, it is easy to tell whether $\mc{M}_{\mbf{1}}/\ul{H} \to \Spd C$ is proper.

\begin{proposition}\label{0.9}
If $H \subset \GL_n(K)$ is any closed subgroup, the structure map $\mc{M}_{\mbf{1}}/\ul{H} \to \Spd C$ is separated; moreover, it is proper if and only if $\GL_n(K)/H$ is compact. In particular, any quotient $\mc{M}_{\mbf{1}}/ \ul{P_{n-d,d}(K)}$ is proper over $\Spd C$.
\end{proposition}

\begin{proof}
For any such quotient, the structure map to $\Spd C$ factors
over a (surjective!) map $q : \mc{M}_{\mbf{1}}/ \ul{H} \to \mb{P}_{C}^{n-1,\lozenge}$ induced by the Gross-Hopkins period map. The pullback of $q$ along the v-cover $\mc{M}_{\mbf{1}} \to \mb{P}_{C}^{n-1,\lozenge}$ is then canonically identified with the projection map $\tilde{q} : \ul{\GL_n(K)/H} \times \mc{M}_{\mbf{1}} \to\mc{M}_{\mbf{1}}$. The latter map is always separated, so $q$ is separated by \cite[Proposition 10.11(ii)]{dia}; since the target of $q$ is separated over $\Spd C$, this shows that the source is too. Likewise, $q$ is quasicompact if and only if $\tilde{q}$ is quasicompact, and the latter clearly holds if and only if $\mathrm{GL}_{n}(K)/H$ is compact.
\end{proof}

We begin by analyzing the general quotients $\mc{M}_{\mbf{1}}/\ul{P_{n-d,d}(K)}$.

\begin{proposition}\label{0.10}
Fix any $1\leq d<n$. Then the diamond quotient $\mc{M}_{\mbf{1}}/\ul{P_{n-d,d}(K)}$ is canonically identified with the functor $X_{n,d}$ on perfectoid spaces over $C$ sending any $S$ to the set of (isomorphism classes of) diagrams 
\[
\mc{O}(1/n) \twoheadrightarrow \mc{E} \hookleftarrow \mc{F}
\]
of vector bundles over $\mc{X}_{S^{\flat}}$ such that $\mc{E} \simeq\mc{O}(1/d)$ and $\mc{F} \simeq \mc{O}^{d}$ at all geometric points and such that $\mathrm{coker}(\mc{F} \to \mc{E}) \simeq \iota_{\ast}W$ for some projective rank one $\mathcal{O}_{S}$-module $W$.
\end{proposition}

\begin{proof}
First, observe that there is a natural map $\mc{M}_{\mbf{1}} \to X_{n,d}$, given by sending any $\{ \alpha : \mc{O}^{n} \to \mc{O}(1/n) \} \in \mc{M}_{\mbf{1}}(S)$ to the diagram 
\[
\mc{O}(1/n) \twoheadrightarrow \mc{O}(1/n) / \alpha(\mc{O}^{n-d}\oplus 0) \hookleftarrow \alpha(\mc{O}^{n}) / \alpha(\mc{O}^{n-d} \oplus 0)
\]
of vector bundles over $\mc{X}_{S^{\flat}}$. For this, observe
that the quotient $\mc{O}(1/n) / \alpha(\mc{O}^{n-d} \oplus 0)$
is isomorphic to $\mc{O}(1/d)$ at all geometric points by Lemma
\ref{0.7}, and the remaining conditions are clearly satisfied. The datum of this diagram only depends on the $\ul{P_{n-d,d}(K)}(S)$-orbit of $\alpha$, so this map factors over a natural transformation $\mc{M}_{\mbf{1}} / \ul{P_{n-d,d}(K)} \to X_{n,d}$, and we claim this transformation is actually an isomorphism. 

\medskip

It clearly suffices to check that $\mc{M}_{\mbf{1}} \to X_{n,d}$
is a $\ul{P_{n-d,d}(K)}$-torsor. For this, let $\mc{O}(1/n) \twoheadrightarrow \mc{E} \hookleftarrow \mc{F}$ be any $S$-point of $X_{n,d}$. Let $\mc{V} \subset \mc{O}(1/n)$ be the rank $n$ sub-vector bundle defined by the cartesian diagram
\[
\xymatrix{\mc{V} \ar[r]^{\gamma} \ar[d] & \mc{F} \ar[d] \\
\mc{O}(1/n) \ar[r] & \mc{E}
}
\]
so $\mc{V} \to \mc{O}(1/n)$ is stably injective and $\mc{O}(1/n) / \mc{V} = \mc{E}/\mc{F} \simeq \iota_{\ast}W$.
Moreover, at any geometric point of $S$, $\mc{V}$ has degree zero and all HN slopes $\leq1/n$, so in fact $\mc{V} \simeq \mc{O}^{n}$
at all geometric points. Now, the ambiguity in lifting our given $S$-point of $X_{n,d}$ to an $S$-point of $\mc{M}_{\mbf{1}}$ is exactly the ambiguity of choosing a trivialization $\mc{O}^{n} \overset{\sim}{\to} \mc{V}$ which maps $\mc{O}^{n-d} \oplus 0$ isomorphically onto $\ker \gamma$, and the space of such trivializations is clearly a $\ul{P_{n-d,d}(K)}$-torsor over $X_{n,d}$, as desired.
\end{proof}

\begin{proposition}\label{0.11}
Fix any $1\leq d<n$. Then the diamond $\mc{M}_{\mbf{1}}/\ul{P_{n-d,d}(K)} \cong X_{n,d}$ is isomorphic to the quotient
\[
\left(\mc{S}\mathrm{urj}(\mc{O}(1/n),\mc{O}(1/d)) \times_{\Spd C} \mb{P}_{C}^{d-1,\lozenge} \right)/\ul{D_{1/d}^{\times}}.
\]
Here $\mc{S}\mathrm{urj}(\mc{O}(1/n),\mc{O}(1/d))$ is the functor on perfectoid spaces over $C$ parametrizing surjective maps $\mc{O}(1/n) \to \mc{O}(1/d)$, and $D_{1/d}$ is the division algebra over $K$ of invariant $1/d$, with $\ul{D_{1/d}^{\times}}$ acting diagonally on the two factors.
\end{proposition}

\begin{proof}
Let $\tilde{X}_{n,d}$ be the $\ul{D_{1/d}^{\times}}$-torsor over $X_{n,d}$ which (in the notation of Proposition \ref{0.10}) parametrizes trivializations $\mc{O}(1/d) \overset{\sim}{\to}\mc{E}$. Then $\tilde{X}_{n,d}$ clearly decomposes as 
\[
\mc{S}\mathrm{urj}(\mc{O}(1/n),\mc{O}(1/d)) \times_{\Spd C} Y
\]
where $Y$ is the functor whose $S$-points parametrize subbundles
$\mc{F} \subset \mc{O}(1/d)$ such that $\mc{O}(1/d)/\mc{F} \simeq\iota_{\ast}W$ for some projective rank one $\mc{O}_{S}$-module $W$. The data of such an $\mc{F}$ is obviously equivalent to the data of a rank one projective $\mc{O}_{S}$-module quotient $\iota^{\ast}\mc{O}(1/d)\to W$: the functor in one direction is obvious, and the functor in the other direction sends $\iota^{\ast} \mc{O}(1/d)\to W$ to 
\[
\ker(\mc{O}(1/d )\to \iota_{\ast}\iota^{\ast}\mc{O}(1/d) \to\iota_{\ast}W).
\]
Finally, $\iota^{\ast}\mc{O}(1/d)$ is \emph{canonically }identified with $\mc{O}_{S}^{d}$. Putting these observations together, $Y$ identifies with the functor sending $S$ to the set of rank one locally free $\mc{O}_{S}$-module quotients $\mc{O}_{S}^{d}=\iota^{\ast}\mc{O}(1/d)\to W$. The latter functor is obviously represented by $\mb{P}_{C}^{d-1,\lozenge}$, as desired.
\end{proof}

\medskip

\begin{proof}[Proof of Theorem \ref{0.5}] Properness follows from Proposition \ref{0.9}. For cohomological smoothness, combining the Proposition \ref{0.11} with \cite[Proposition 24.2]{dia} reduces us to showing that 
\[
\mc{S}\mathrm{urj}(\mc{O}(1/n),\mc{O}(1/d)) \times_{\Spd C}\mb{P}_{C}^{d-1,\lozenge} \to \Spd C
\]
is cohomologically smooth. This reduces to the smoothness of each
factor over $\Spd C$. The projective space factor is immediately handled by \cite[Proposition 24.4]{dia}. For the first factor, we note that $\mc{S}\mathrm{urj}(\mc{O}(1/n),\mc{O}(1/d))$ is an open subfunctor of $\mc{H}^{0}(\mc{O}(1/d) \otimes \mc{O}(-1/n))$, cf. \cite[Proposition 3.3.6]{bfh}. Since $\mc{O}(1/d) \otimes \mc{O}(-1/n)$ has slopes strictly between $0$ and $1$, the latter functor is representable by an open perfectoid ball in $n-d$ variables over $C$, so now smoothness follows from \cite[Proposition 24.1]{dia}.

\medskip

Finally, suppose that $\mc{M}_{\mbf{1}}/\ul{P_{n-d,d}(K)}$ is a perfectoid space. By Proposition \ref{0.11}, we have a $\ul{D_{1/d}^{\times}}$-torsor
\[
\mc{S}\mathrm{urj}(\mc{O}(1/n),\mc{O}(1/d)) \times_{\Spd C}\mb{P}_{C}^{d-1,\lozenge} \to \mc{M}_{\mbf{1}}/\ul{P_{n-d,d}(K)}.
\]
By assumption, the target is perfectoid, so then the source is perfectoid as well by \cite[Proposition 10.11]{dia}. Intuitively, we now expect a contradiction if $d>1$, because the projective space factor should contribute ``non-perfectoid directions'' to the source. To make this precise, choose some perfectoid field $C'/C$ and a map $\Spd C' \to \mc{S}\mathrm{urj}(\mc{O}(1/n),\mc{O}(1/d))$. We've already observed that $\mc{S}\mathrm{urj}(\mc{O}(1/n),\mc{O}(1/d))$ is perfectoid, so
\[
\Spd C' \times_{\mc{S}\mathrm{urj}(\mc{O}(1/n),\mc{O}(1/d))}\left( \mc{S}\mathrm{urj}(\mc{O}(1/n),\mc{O}(1/d)) \times_{\Spd C}\mb{P}_{C}^{d-1,\lozenge} \right)
\]
is a fiber product of perfectoid spaces, and thus is perfectoid. On
the other hand, this fiber product is just $\mb{P}_{C'}^{d-1,\lozenge}$. Putting things together, we've shown that if $\mc{M}_{\mbf{1}}/\ul{P_{n-d,d}(K)}$ is perfectoid, then $\mb{P}_{C'}^{d-1,\lozenge}$ is necessarily perfectoid, which forces $d=1$, as desired.
\end{proof}

\begin{proof}[Proof of Theorem \ref{0.1}] Specializing Proposition \ref{0.11} to the situation where $d=1$, we get a canonical identification
\[
\mc{M}_{\mbf{1}}/\ul{P(K)} \cong \mc{S}\mathrm{urj}(\mc{O}(1/n),\mc{O}(1))/\ul{K^{\times}}.
\]
This is nothing more than the functor parametrizing quotients $\mc{O}(1/n) \twoheadrightarrow \mc{L}$ where $\mc{L}$ is a line bundle of degree one. It remains to identify this functor with the functor parametrizing subbundles $\mc{E} \subset \mc{O}(1/n)$ as specified in Theorem \ref{0.1}. 

\medskip

For this, note that sending any such $\mc{E} \subset \mc{O}(1/n)$ to the quotient $\mc{O}(1/n) \twoheadrightarrow \mc{O}(1/n)/\mc{E}$ defines a natural transformation in one direction, since $\mc{O}(1/n)/\mc{E}$ is a line bundle of degree one by Lemma \ref{0.7}. We also have a transformation in the other direction, sending any $q : \mc{O}(1/n) \twoheadrightarrow \mc{L}$ to the inclusion $\ker q \subset \mc{O}(1/n)$: one easily checks that, at any geometric point, $\ker q$ has rank $n-1$, degree zero, and all Harder--Narasimhan slopes $\leq 1/n$, so $\ker q \simeq \mc{O}^{n-1}$ at any geometric point. These two natural transformations are mutually inverse to each other, as desired.
\end{proof}

\subsection{The case $n=2$}

In this section we prove Theorem \ref{0.3} and Corollary \ref{0.4}. In particular, we assume $n=2$ throughout. As in the body of the paper, fix a uniformizer $\vp \in \mc{O}_{K}$. Let $E$ be the unramified quadratic extension of $K$, and let $\mc{G}=\mc{G}_{E}$ be the unique Lubin--Tate formal $\mc{O}_{E}$-module for which multiplication by $\vp$ is given by the polynomial $f(T)=T^{q^{2}}+\vp T$. Let $\tilde{E}/K$ be the completion of the extension obtained by adjoining all $\vp$-division points of $\mc{G}$ to $E$. By Lubin--Tate theory, $\tilde{E}$ is (the completion of an extension which is) Galois over $E$ with Galois group $\mc{O}_{E}^{\times}$, and $\tilde{E}$ is a perfectoid field.

\begin{lemma}\label{0.12}
The fixed field $L=\tilde{E}^{\mc{O}_{K}^{\times}}$ is a perfectoid
field.
\end{lemma}

\begin{proof}
By the basic definitions, $L$ is the completion of a Galois extension of $E$ with Galois group $\mc{O}_{E}^{\times}/\mc{O}_{K}^{\times}$. This is an abelian $p$-adic Lie group of dimension $[K:\Qp]>0$, so $L$ is perfectoid by a theorem of Sen \cite{sen}. (Alternately, up to a finite extension, $L$ is the completion of a compositum of totally ramified $\Zp$-extensions of $E$, so we could appeal to Tate's original results \cite{tate}.)
\end{proof}

Let $\tilde{\mc{G}} = \varprojlim_{[\vp]} \mc{G}$ be the universal cover of $\mc{G}$ (as in \cite[\S 3.1]{sw}), and let $\tilde{\mc{G}}_{0}$ be its reduction modulo $\vp$. As in \cite[\S 3.5]{wein}, there is an identification $\tilde{\mc{G}}_{0}=\Spf \F_{q^{2}}[[T^{1/p^{\infty}}]]$. This is a formal $E$-vector space in the category of formal schemes over $\F_{q^{2}}$.

\medskip

To relate this object to vector bundles, let $Y$ be the functor on perfectoid spaces over $\F_{q^{2}}$ sending any $S$ to $H^{0}(\mc{X}_{S,K},\mc{O}(1/2)) = H^{0}(\mc{X}_{S,E},\mc{O}(1))$, and let $Y^{\times}\subset Y$ be the open subfunctor of nowhere-vanishing
sections. By Theorem \ref{0.1}, there is a natural identification
\[
\mc{M}_{\mbf{1}}/\ul{P(K)} \cong Y^{\times}/\ul{K^{\times}} \times_{\Spd \F_{q^{2}}} \Spd C.
\]

\begin{proposition}\label{0.13}
There are compatible $K^{\times}$-equivariant isomorphisms $Y \cong\tilde{\mc{G}}_{0}$ and $Y^{\times} \cong \tilde{\mc{G}}_{0} \smallsetminus \{0\}$.
\end{proposition}

\begin{proof}
The first isomorphism follows from the fact that $Y$ and $\tilde{\mc{G}}_{0}$ are both naturally identified with the functor $\mbf{B}_{\mathrm{crys},E}^{+,\varphi_{q^{2}}=\vp}$. For $Y$ this identification is immediate from the definition of $\mc{O}(1)$ over $\mc{X}_{S,E}$ as the descent of a $\varphi_{q^{2}}$-equivariant bundle on the usual cover $\mc{Y}_{S,E} \to \mc{X}_{S,E}$. For $\tilde{\mc{G}}_{0}$ this identification follows from \cite{sw}, cf. Theorem 4.1.4 and the first line in the proof of Proposition 6.3.9. In fact, this identification can be given by an explicit formula: an $(R,R^+)$-point of $\tilde{\mc{G}}_{0}$ is the same as an element $x\in R^{\circ\circ}$, and we map $x$ to the element $\mathrm{log}_{\mc{G}}(\{x\})$, where $\{x\}=\underset{n\to\infty}{\lim}[\vp^n](\varphi_{q^{2}}^{-n}(\tilde{x}))$ with $\tilde{x}\in W_{\mc{O}_E}(R^+) = W(R^+) \otimes_{W(\F_{q^{2}})} \mc{O}_{E}$ any lift of $x$. The second isomorphism is then immediate.
\end{proof}

On the other hand, we have

\begin{proposition}\label{0.14}
There are compatible $E^{\times}$-equivariant isomorphisms $\tilde{\mc{G}}_{0} \cong \Spa \mc{O}_{\tilde{E}}^{\flat}$ and $\tilde{\mc{G}}_{0} \smallsetminus \{0\} \cong \Spa \tilde{E}^{\flat}$, where on the right-hand sides $\mc{O}_{E}^{\times} \subset E^{\times}$ acts through its natural identification with $\Gal(\tilde{E}/E)$ and $\vp$ acts as the $q^2$-power Frobenius.
\end{proposition}

\begin{proof}
This follows from \cite[Proposition 3.5.3]{wein}.
\end{proof}

Putting these two propositions together, we get a $K^{\times}$-equivariant isomorphism
\[
Y^{\times} \times_{\Spd \F_{q^2}} \Spd C \cong \Spd \tilde{E}^{\flat} \times_{\Spd \F_{q^2}} \Spd C.
\]
Passing to the quotient by the action of $\mc{O}_{K}^{\times}$
gives
\begin{align*}
Y^{\times}/\ul{\mc{O}_{K}^{\times}} \times_{\Spd \F_{q^2}} \Spd C & \cong (\Spd \tilde{E}^{\flat})/\ul{\mc{O}_{K}^{\times}} \times_{\Spd \F_{q^2}} \Spd C \\
 & \cong \Spd (\tilde{E}^{\flat})^{\mc{O}_{K}^{\times}} \times_{\Spd \F_{q^2}} \Spd C \\
 & \cong \Spd L^{\flat} \times_{\Spd \F_{q^2}} \Spd C
\end{align*}
where in the second line we've used Lemma \ref{0.12}. Note that this diamond is the fiber product of two characteristic $p$ perfectoid spaces over a \emph{discrete} field. Nevertheless we have the following result.

\begin{proposition}\label{0.15}
Let $k$ be a discrete field, and let $X$ and $Y$ be perfectoid spaces over $k$. Then the product $X\times Y$ is representable by a perfectoid space over $k$, where the product is taken in the category of sheaves of sets on $\mathrm{Perf}_{k}$.
\end{proposition}

By this result, $\Spd L^{\flat} \times_{\Spd \F_{q^2}} \Spd C$ is representable by a perfectoid space over $\F_{q^2}$, which moreover comes equipped with a canonical map to $\Spd C = \Spa C^{\flat}$. Moreover, writing $\varphi : \Spd L^{\flat} \to \Spd L^{\flat}$ for the $q^2$-power Frobenius, we easily see that $\varphi \times\mathrm{id}$ acts properly discontinuously on this product, so the quotient
\[
(\Spd L^{\flat} \times_{\Spd \F_{q^2}} \Spd C) / (\varphi\times\mathrm{id})^{\Z}
\]
is representable by a perfectoid space over $\F_{q^2}$ with a map to $\Spd C$. This has a unique untilt to a perfectoid space over $C$. On the other hand, summarizing the analysis above, we have canonical
isomorphisms
\[
(\Spd L^{\flat} \times_{\Spd \F_{q^2}} \Spd C)/ (\varphi \times\mathrm{id})^{\Z} \cong Y^{\times}/\ul{K^{\times}} \times_{\Spd \F_{q^2}} \Spd C \cong \mc{M}_{\mbf{1}}/\ul{P(K)},
\]
so $\mc{M}_{\mbf{1}}/\ul{P(K)}$ is a perfectoid space, as desired. This finishes the proof of Theorem \ref{0.3} and Corollary \ref{0.4}.

\bibliography{LTII}
\bibliographystyle{alpha}

\end{document}